\documentclass[a4paper,reqno,11pt]{amsart}

\usepackage{amssymb}
\usepackage{amstext}
\usepackage{amsmath}
\usepackage{amscd}
\usepackage{amsthm}
\usepackage{amsfonts}
\usepackage{enumerate}
\usepackage{graphicx}
\usepackage{color}
\usepackage{here} 
\usepackage{bm} 
\usepackage{mathrsfs}
\usepackage{mathtools}
\usepackage{latexsym}
\usepackage{booktabs}
\usepackage{fancyhdr}
\usepackage{subcaption}
\usepackage{tikz}
\usepackage{tikz-cd}
\usetikzlibrary{arrows}
\usetikzlibrary{decorations.markings}
\usetikzlibrary{patterns,decorations.pathreplacing}

\makeatletter
  
\@addtoreset{equation}{section}
\@namedef{subjclassname@2020}{\textup{2020} Mathematics Subject Classification}
\makeatother

\newcommand{\calA}{\mathcal{A}}
\newcommand{\calB}{\mathcal{B}}

\newcommand{\calF}{\mathcal{F}}
\newcommand{\calG}{\mathcal{G}}

\newcommand{\calO}{\mathcal{O}}
\newcommand{\calP}{\mathcal{P}}
\newcommand{\calM}{\mathcal{M}}
\newcommand{\calS}{\mathcal{S}}

\newcommand{\ZZ}{\mathbb{Z}}
\newcommand{\QQ}{\mathbb{Q}}
\newcommand{\RR}{\mathbb{R}}

\newcommand{\kk}{\Bbbk}

\newcommand{\ab}{\mathbf{a}}

\newcommand{\bb}{\mathbf{b}}
\newcommand{\cb}{\mathbf{c}}
\newcommand{\eb}{\mathbf{e}}
\newcommand{\fb}{\mathbf{f}}

\newcommand{\ib}{\mathbf{i}}
\newcommand{\jb}{\mathbf{j}}

\newcommand{\tb}{\mathbf{t}}
\newcommand{\ub}{\mathbf{u}}
\newcommand{\vb}{\mathbf{v}}
\newcommand{\wb}{\mathbf{w}}

\newcommand{\xb}{\mathbf{x}}
\newcommand{\yb}{\mathbf{y}}
\newcommand{\zb}{\mathbf{z}}

\newcommand{\Hom}{\operatorname{Hom}}

\newcommand{\ang}[1]{\langle #1 \rangle}

\newcommand{\Gro}{Gr\"{o}bner }

\def\opn#1#2{\def#1{\operatorname{#2}}} 
\opn\Cl{Cl} \opn\conv{conv} \opn\deg{deg} \opn\rank{rank} \opn\Spec{Spec} \opn\Stab{Stab} \opn\aff{aff} \opn\div{div} \opn\GL{GL}
\opn\cone{cone} \opn\End{End} \opn\Hom{Hom} \opn\mod{mod} \opn\gldim{gldim} \opn\pdim{pdim} \opn\diag{diag} \opn\vert{vert}
\opn\Block{Block} \opn\Pyr{Pyr} \opn\max{max} \opn\min{min} \opn\ini{in} \opn\rev{rev} \opn\ker{ker} \opn\lat{lat} \opn\pull{pull} \opn\rev{rev} \opn\Sym{Sym} \opn\supp{supp} \opn\int{int} \opn\star{star} \opn\sign{sign} \opn\Gale{Gale}

\newtheorem{thm}{Theorem}[section]

\newtheorem{lem}[thm]{Lemma}
\newtheorem{prop}[thm]{Proposition}
\newtheorem{p}{Problem}
\newtheorem{q}{Question}

\theoremstyle{definition}

\newtheorem{ex}[thm]{Example}

\theoremstyle{remark}

\begin{document}

\title{Toric rings of $(0,1)$-polytopes with small rank}
\author{Koji Matsushita}

\address{Department of Pure and Applied Mathematics, Graduate School of Information Science and Technology, Osaka University, Suita, Osaka 565-0871, Japan}
\email{k-matsushita@ist.osaka-u.ac.jp}

\subjclass[2020]{
Primary 13F65; 
Secondary 13C20, 
52B20} 

\keywords{Toric rings, $(0,1)$-polytopes, Divisor class groups}

\maketitle

\begin{abstract} 
The rank of a $d$-dimensional polytope $P$ is defined by $F-(d+1)$, where $F$ denotes the number of facets of $P$.
In this paper, We focus on the toric rings of $(0,1)$-polytopes with small rank.
We study their normality, the torsionfreeness of their divisor class groups and the classification of their isomorphism classes.
\end{abstract}

\bigskip


\section{Introduction}
Throughout the present paper, let $\kk$ be an algebraically closed field of characteristic $0$, for simplicity.
\subsection{Backgrounds}
For an integral polytope $P\subset \RR^d$, let $\kk[P]$ denote the toric ring of $P$, i.e., 
$$\kk[P]=\kk[t_1^{v_1}\cdots t_d^{v_d}t_0 : (v_1,\ldots,v_d)\in P\cap\ZZ^d].$$
Toric rings of integral polytopes are homogenous affine semigroup rings and their algebraic properties have been well investigated.
In particular, normality provides various important properties to toric rings and some algebraic objects can be described in the terms associated with polytopes.
In fact, it is well known that normal toric rings are Cohen-Macaulay for any field $\kk$ (\cite{Ho}) and that the canonical module of the normal toric ring is described as the ideal generated by the monomials corresponding to the elements in the interior of $\ZZ_{\ge 0}\calA(P)$ (\cite{S78}), where $\calA(P)=\{(\vb,1) \in \ZZ^{d+1} : \vb\in P\cap \ZZ^d \}$ and 
$\ZZ_{\ge 0}\calA(P)=\{a_1\xb_1+\cdots +a_n\xb_n : \xb_1,\ldots,\xb_n\in \calA(P), a_1,\ldots,a_n \in \ZZ_{\ge 0} \}$.
Moreover, the toric ring $\kk[P]$ is normal if and only if it is a Krull ring (cf. \cite[Theorem~9.8.13]{Villa}).
In this case, the divisor class group of $\kk[P]$ can be easily computed by using the generators and the support forms of $\ZZ_{\ge 0}\calA(P)$ (see Section~\ref{subsec:toric}).

The divisor class group $\Cl(\kk[P])$ of a normal toric ring $\kk[P]$, which is a finitely generated abelian group, is one of the most interesting invariants.
Its elements are identified with the isomorphism classes of the divisorial ideals of $\kk[P]$.
Recently, conic divisorial ideals, which are a certain class of divisorial and a special kind of maximal Cohen-Macaulay modules of rank one, are well studied (see, e.g., \cite{Bru, BG1}).
Indeed, they play beneficial roles in the theory of non-commutative algebraic geometry as well as commutative rings.
For example, we can construct non-commutative crepant resolutions (NCCRs, for short), which were introduced by Van den Bergh (\cite{VdB}), by considering the endomorphism ring of the direct sum of some conic divisorial ideals.
In fact, they are constructed for many classes of toric rings (see, e.g., \cite{HM, HN, M2, N, SpVdB, VdB}, and so on).
However, in general, the existence of NCCRs for Gorenstein toric rings is still open.
In solving this problem, it is natural to consider the case of the toric rings whose divisor class groups have small rank.
On the other hand, it was shown in \cite{DITW} that if a Cohen-Macaulay normal domain $R$ has an NCCR, then $R$ is $\QQ$-Gorenstein.
In particular, if the divisor class group of $R$ is torsionfree and $R$ is not Gorenstein, then $R$ does not have an NCCR.
Therefore, it is important to decide whether divisor class groups are torsionfree.

The convex hull of a finite set of $(0,1)$-vectors is called a $(0,1)$-polytope.
It arises from various combinatorial objects such as partially ordered sets (posets, for short), graphs and matroids.
For instance, the following are families of $(0,1)$-polytopes and their toric rings are well studied:
\begin{itemize}
\item Order polytopes (see \cite{S86} and Section~\ref{subsection:01});
\item Stable set polytopes (\cite{Ch}) $\supset$ Chain polytopes (\cite{S86});
\item Edge polytopes (see e.g., \cite{OH98,SVV});
\end{itemize}
Of course, these are not the only ones, there are much more; cut polytopes (\cite{DL}), perfect matching polytopes (\cite{E}), matroid polytopes (\cite{O}), and so on.
In addition, compressed polytopes introduced by Stanley \cite{S80} are a natural class of integral polytopes and include some of the $(0,1)$-polytopes mentioned above. 
It is known that the toric ring of a compressed polytope is isomorphic to that of a $(0,1)$-polytope (\cite[Theorem~2.4]{Su}).

By studying the commutative ring-theoretic properties of their toric rings, we can solve problems related to underlying combinatorial objects.

\medskip

In the present paper, we focus on the \textit{rank} of an integral polytope $P$, which is defined by
$$\rank P= F-(\dim P +1),$$
where $F$ denotes the number of facets of $P$.
Note that $\rank P$ is a nonnegative integer.
Actually, the rank of $P$ coincide with that of $\Cl(\kk[P])$ if $P$ is normal (see Section~\ref{subsec:toric}).
Namely, we deal with toric rings whose divisor class groups have small rank.

The goal of the present paper is to analyze the normality and torsionfreeness of $(0,1)$-polytopes, and to classify the isomorphism classes of their toric rings in the case where the underlying polytopes have small rank.



\subsection{Normality}
We say that an integral polytope $P$ is normal if so is $\kk[P]$.
The normality of an integral polytope $P$ can be characterized in terms of the lattice points generated by $\calA(P)$ (see Section~\ref{subsec:toric}).
In general, it is not easy to determine if $P$ is normal or not,
but there is a useful sufficient condition for $P$ to be normal; $P$ is normal if $P$ possesses a unimodular triangulation (cf. \cite[Corollary~4.12]{HHO}).
This implies that order polytopes and stable set polytopes of perfect graphs including chain polytopes are normal since these are compressed polytopes (see \cite{OH01}), which possess a unimodular triangulation.
Meanwhile, it is known that the edge polytope of a graph $G$ is normal if and only if $G$ satisfies the odd cycle condition (\cite{OH98,SVV}).

As this shows, $(0,1)$-polytopes are not necessarily normal.
However, non-normal $(0,1)$-polytopes seem to have large rank.
In fact, we can observe that the rank of a non-normal edge polytope is at least 4 by \cite[Proposition~4.4]{HM3} and some arguments. 
Therefore, we are interested in the normality of integral polytopes with small rank and consider the following problem:

\begin{p}\label{p:normal}
Examine the normality of $(0,1)$-polytopes with small rank.
\end{p}



\subsection{Torsionfreeness}
We say that a normal integral polytope $P$ is torsionfree if so is $\Cl(\kk[P])$.
The torsionfreeness of a normal integral polytope $P$ can be determined by computing the Smith normal form of the matrix constructed from the lattice points in $P$ and the facet defining inequalities of $P$ (see Section~\ref{subsec:toric}).
In \cite{M}, a sufficient condition for $P$ to be torsionfree is given.
As a corollary, it follows that compressed polytopes are torsionfree (\cite[Corollary~3.4]{M}), and so are order polytopes and stable set polytopes of perfect graphs.
Moreover, normal edge polytopes are also torsionfree (\cite[Theorem~3.6]{HM3}).
Furthermore, recently, the divisor class groups of toric face rings, which can be regarded as the toric rings of $(0,1)$-polytopes arising from simplicial complexes, are investigated in \cite{HHMQ}.
It is shown that they are torsionfree if their underlying toric face rings are normal (\cite[Corollary~1.8]{HHMQ}). 

These facts pose us the following problem:

\begin{p}\label{p:torsion}
Are normal $(0,1)$-polytopes always torsionfree?
\end{p}

\medskip


\subsection{Classification}
The classification of the isomorphism classes and relationships among certain classes of the toric rings whose divisor class groups have small rank have been studied in \cite{HM3} as follows. 
Let ${\bf Order}_n$ (resp. ${\bf Stab}_n$, ${\bf Edge}_n$) denote 
the sets of isomorphic classes of the toric rings of order polytopes (resp. stable set polytopes of perfect graphs, normal edge polytopes) with rank $n$.
Then, the following relationships hold:
\begin{itemize}
\setlength{\parskip}{0pt} 
\setlength{\itemsep}{3pt}
\item ${\bf Order}_n={\bf Stab}_n={\bf Edge}_n$ if $n\le 1$.
\item ${\bf Stab}_2 \cup {\bf Edge}_2={\bf Order}_2$ and no inclusion between ${\bf Stab}_2$ and ${\bf Edge}_2$;
\item there is no inclusion among ${\bf Order}_n$, ${\bf Stab}_n$ and ${\bf Edge}_n$ if $n\ge 3$.
\end{itemize}
In particular, the relationship ${\bf Stab}_n\cup {\bf Edge}_n \subset {\bf Order}_n$ holds if $n\le 2$.

Now, we consider a generalization of this result.
Let ${\bf (0,1)}_n$ denote the sets of isomorphic classes of the toric rings of $(0,1)$-polytopes with rank $n$.
Clearly, the relationship ${\bf Order}_n\cup{\bf Stab}_n\cup{\bf Edge}_n \subset {\bf (0,1)}_n$ holds.
It follows from the third relationship of the above three families that one has ${\bf Order}_n \subsetneq {\bf (0,1)}_n$ if $n\ge 3$.
On the other hand, in the case $n\le 2$, the relationship between ${\bf (0,1)}_n$ and ${\bf Order}_n$ is very close.

\begin{p}\label{p:class}
Determine the set ${\bf (0,1)}_n$.
Also, does the relationship ${\bf (0,1)}_n={\bf Order}_n$ hold if $n\le 2$?
\end{p}

In classifying isomorphism classes of the toric rings of $(0,1)$-polytopes, 
to consider Gale-diagrams, which is an effective way to investigate the combinatorial types of polytopes with ``few vertices", is compatible with our study.
In this paper, we treat polytopes with small rank, i.e., ``few facets", so we can apply the method of Gale-diagrams by considering the dual polytopes.
We present the following problem:

\begin{p}\label{p:comb}
Does the combinatorial equivalence of two $(0,1)$-polytopes imply the isomorphism of their toric rings?
\end{p}



\subsection{Results}
Let $P$ be a $(0,1)$-polytope.
We give complete or partial answers to Problems~\ref{p:normal}, \ref{p:torsion}, \ref{p:class} and \ref{p:comb} in each case; 
\begin{center}
(r1) \; $\rank P=0$ or $1$, \quad \quad  (r2) \; $\rank P=2$, \quad \quad  (r3) \; $\rank P\ge 3$.
\end{center}

First, we provide the complete answers to the above problems in case (r1).
Actually, most of this result has been already given in \cite{M}:

\begin{thm}[Proposition~\ref{main:rank0} and Theorem~\ref{main:rank1}]
Let $P$ be a $(0,1)$-polytope with $\rank P \le 1$.
Then, $P$ is normal and torsionfree.
Moreover, the following relationships hold:
$${\bf (0,1)}_0={\bf Order}_0=\{\kk[x_1,\ldots,x_n] : n\in \ZZ_{>0}\} \text{ and }$$
$${\bf (0,1)}_1={\bf Order}_1=\{(\kk[x_1,\ldots,x_{n+1}]\#\kk[y_1,\ldots,y_{m+1}])\otimes_{\kk}\kk[z_1,\ldots,z_{l-1}] : n,m,l\in \ZZ_{>0}\}.$$
Furthermore, for two $(0,1)$-polytopes $P_1$ and $P_2$ with $\rank P_i\le 1$, these polytopes have the same combinatorial type if and only if $\kk[P_1]\cong \kk[P_2]$.
\end{thm}

Next, we discuss case (r2).
Before that, we need to introduce a new family of $(0,1)$-polytopes $P_{n_1,\ldots,n_k}$.

Let $\calB_n$ denote the standard basis of $\RR^n$.
For $n_1,\ldots,n_k \in \ZZ_{>0}$, the $(0,1)$-polytope $P_{n_1,\ldots,n_k}$ is defined as the convex hull of the origin, $\calB_d$, where $d:=n_1+\cdots+n_k$, and the product $\calB_{n_1}\times \cdots \times\calB_{n_k}$.

In Section~\ref{sec:new}, we study several properties of $P_{n_1,\ldots,n_k}$.
Specifically, we describe the facet defining inequalities of $P_{n_1,\ldots,n_k}$ (Proposition~\ref{prop:facet}).
Moreover, we show that $\kk[P_{n_1,\ldots,n_k}]$ possesses a squarefree
initial ideal (Proposition~\ref{prop:gro}).
Furthermore, the following facts hold (Theorem~\ref{thm:new_pro}):
\begin{itemize}
\item 
$P_{n_1,\ldots,n_k}$ has the integer decomposition property, and hence it is normal.
\item $\kk[P_{n_1,\ldots,n_k}]$ is Gorenstein if and only if $n_1=n_2=\cdots =n_k$.
\item $\Cl(\kk[P_{n_1,\ldots,n_k}]) \cong \ZZ^{k-1}$.
\item If $k\ge 3$, then $\kk[P_{n_1,\ldots,n_k}]\notin {\bf Order}_{k-1}$ for any $n_1,\ldots,n_k\in \ZZ_{>0}$.
\end{itemize}

We provide various observations on Problems~\ref{p:class} and \ref{p:comb} in case (r2):
\begin{itemize}
\item The above facts imply that $\rank P_{n_1,n_2,n_3}=2$ and $\kk[P_{n_1,n_2,n_3}]\notin {\bf Order}_2$ for any $n_1,n_2,n_3 \in \ZZ_{>0}$, this is a counterexample to the equality ${\bf (0,1)}_2={\bf Order}_2$ of Problem~\ref{p:class}.
\item We study the combinatorial types of the dual polytopes of $\calO_{\Pi_1},\ldots,\calO_{\Pi_4}$ (these are the order polytopes with rank 2) and $P_{n_1,n_2,n_3}$ by drawing their standard Gale-diagrams (see Figures~\ref{type1},\ref{type2},\ref{type3},\ref{type4} and \ref{type5}).
In particular, Problem~\ref{p:comb} has positive answers for $(0,1)$-polytopes that the standard Gale-diagrams of their dual polytopes coincide with $\Gale_1$ or $\Gale_2$. (Theorems~\ref{thm:type1} and \ref{thm:type2})
\end{itemize}

Finally, we give answers to the above problems in case (r3), although they are negative:

\begin{thm}[Propositions~\ref{prop:nonnormal}, \ref{prop:nontorsion} and \ref{prop:noniso}]
For any positive integer $r\ge 3$, there exist non-normal $(0,1)$-polytope $P$ and non-torsionfree normal $(0,1)$-polytope $P'$ with $\rank P=\rank P'=r$, respectively.
Moreover, there exist two $(0,1)$-polytopes $Q$ and $Q'$ with the same combinatorial type and $\rank Q=\rank Q'=r$ such that their toric rings are not isomorphic to each other.
\end{thm}



\subsection{Organization}
In Section~\ref{sec:pre}, we recall the definitions and notation of toric rings of integral polytopes and give an algorithm to compute divisor class groups of toric rings of normal integral polytopes.
We also recall the toric rings of certain $(0,1)$-polytopes, especially the order polytopes, and recall its several properties.
Moreover, we recall the notion of Gale-diagrams and provide a method to obtain the Gale-diagram of the dual polytope of a given polytope.
In Section~\ref{sec:new}, we introduce a new family of $(0,1)$-polytopes and study its properties.
In Section~\ref{sec:ans}, we give complete or partial answers to Problems~\ref{p:normal}, \ref{p:torsion}, \ref{p:class} and \ref{p:comb} in each case; $\rank P=0$ or $1$, $\rank P=2$, and $\rank P\ge 3$.

\medskip

\subsection*{Acknowledgement} 
The author would like to thank Akihiro Higashitani for a lot of his helpful comments and instructive discussions.
The author is partially supported by Grant-in-Aid for JSPS Fellows Grant JP22J20033.


\bigskip

\section{Preliminaries}\label{sec:pre}

\subsection{Toric rings of integral polytopes and their divisor class groups}\label{subsec:toric}
In this subsection, we recall the definitions and notation of polytopes and toric rings. We refer the readers to e.g., \cite{BG2} or \cite{Villa}, for the introduction.

\smallskip

First, we introduce toric rings of integral polytopes.
An \textit{integral polytope} $P \subset \RR^d$ is a polytope whose vertices sit in $\ZZ^d$. 
For an integral polytope $P \subset \RR^d$, we define $\phi_P$ as the morphism of $\kk$-algebras:
$$\phi_P : \kk[x_\vb : \vb\in P\cap \ZZ^d] \to \kk[t_0,t_1^{\pm 1},\ldots,t_d^{\pm 1}], \text{ induced by } \phi_P(x_\vb)=\tb^{\vb}t_0,$$

\noindent where $\tb^{\vb}=t_1^{v_1}\cdots t_d^{v_d}$ for $\vb=(v_1,\ldots,v_d) \in \ZZ^d$. 
Then, the kernel of $\phi_P$, denoted by $I_P$, is called the \textit{toric ideal} of $P$. 
Moreover, the image of $\phi_P$, denoted by $\kk[P]$, is called  the \textit{toric ring} of $P$. 
Note that $\kk[P]\cong \kk[x_\vb : \vb\in P\cap\ZZ^d]/I_P$.

It is well known that the toric ideal $I_P$ is generated by homogeneous binomials. 
The toric ring $\kk[P]$ is a standard graded $\kk$-subalgebra of $\kk[t_0,t_1^{\pm 1},\ldots,t_d^{\pm 1}]$ by setting $\deg(\tb^\vb t_0)=1$ for each $\vb \in P \cap \ZZ^d$. 
The Krull dimension of $\kk[P]$ is equal to the dimension of $P$ plus $1$.

For an integral polytope $P\in \RR^d$, let
$\calA(P)=\{(\vb,1) \in \ZZ^{d+1} : \vb\in P\cap \ZZ^d \}$ and 
let $\ZZ_{\ge 0}\calA(P)=\{a_1\xb_1+\cdots +a_n\xb_n : \xb_1,\ldots,\xb_n\in \calA(P), a_1,\ldots,a_n \in \ZZ_{\ge 0} \}$ (we define $\ZZ\calA(P)$ and $\RR_{\ge 0}\calA(P)$ analogously).

Then, the following facts hold:
\begin{itemize}
\item $P$ is normal (i.e., $\kk[P]$ is normal) if and only if $\ZZ_{\ge 0}\calA(P)=\ZZ\calA(P)\cap \RR_{\ge 0}\calA(P)$ (cf. \cite[Theorem~9.8.13]{Villa}).
\item Suppose that $\ZZ\calA(P)=\ZZ^{d+1}$. Then, $P$ is normal if and only if $P$ has the \textit{integer decomposition property} (\textit{IDP}, for short), that is, for any $n\in \ZZ_{>0}$ and any $\ab \in nP\cap \ZZ^d$, there exist $\ab_1,\ldots,\ab_n \in P\cap \ZZ^d$ such that $\ab=\ab_1+\cdots +\ab_n$ (cf.\cite[Theorem~4.7]{HHO}).
\end{itemize}

\bigskip

We also define the \textit{product} of two polytopes $P\subset \RR^d$ and $Q\subset \RR^e$ as
$$P\times Q=\{(\xb,\yb) : \xb\in P,\yb\in Q\}\subset \RR^{d+e}.$$
We can see that $P\times Q$ is a polytope of dimension $\dim(P) + \dim(Q)$, whose nonempty faces
are the products of nonempty faces (including itself) of $P$ and $Q$.
In particular, the number of facets of $P\times Q$ is equal to $|\Psi(P)|+|\Psi(Q)|$, where $\Psi(P)$ denotes the set of facets of $P$.
Therefore, we have $\rank P\times Q= \rank P + \rank Q +1$.

The toric ring of the product of two integral polytopes corresponds to the ``Segre product" of these toric rings; let $P_1$ and $P_2$ be two integral polytopes, then $\kk[P_1\times P_2]$ is isomorphic to the Segre product of $\kk[P_1]$ and $\kk[P_2]$.
Here, for two standard $\kk$-algebras $R=\bigoplus_{n\ge 0}R_n$ and $S=\bigoplus_{n\ge 0}S_n$, we define their \textit{Segre product} $R\# S$ as the graded $\kk$-algebra:
$$R\#S=(R_0\otimes_{\kk}S_0)\oplus(R_1\otimes_{\kk}S_1)\oplus\cdots \subset R\otimes_{\kk}S.$$



\medskip

Next, we give an algorithm to compute divisor class groups of toric rings of normal integral polytopes.
We use theories in \cite[Section 9.8]{Villa}.

Let $\langle \cdot , \cdot \rangle$ denote the natural inner product of $\RR^d$. For $\ab \in \RR^d$ and $b \in \RR$, 
we denote by $H^{+}(\ab ; b)$ (resp. $H(\ab ; b)$) a closed half-space $\{ \ub \in \RR^d : \langle \ub, \ab\rangle + b\ge 0\}$ 
(resp. an affine hyperplane $\{ \ub \in \RR^d : \langle \ub , \ab \rangle +b= 0\}$). 
In particular, we denote the linear hyperplane $H(\ab;0)$ by $H_\ab$.

For each $F\in \Psi(P)$, there exist a vector $\ab_F \in \QQ^d$ and a rational number $b_F$ 
with the following conditions:
\begin{itemize}
\setlength{\parskip}{0pt} 
\setlength{\itemsep}{3pt}
\item[(d1)] $H(\ab_F ; b_F)$ is a support hyperplane associated with $F$ and $P \subset H^{+}(\ab_F;b_F)$; 
\item[(d2)] $d_F(\vb) \in \ZZ$ for any $\vb \in P \cap \ZZ^d$; 
\item[(d3)] $\sum_{v\in P\cap \ZZ^d}d_F(\vb)\ZZ = \ZZ$,
\end{itemize}
where $d_F(\vb)=\langle \vb, \ab_F \rangle + b_F$ for $\vb \in \ZZ^d$. We can see that $d_F(\vb)$ for $\vb \in P\cap\ZZ^d$ is independent of the choice of $\ab_F$ and $b_F$.

In this situation, $P$ has the irreducible representation:
\begin{align}\label{repre}
\RR_{\ge 0}\calA(P)=\aff(\calA(P))\cap\left(\bigcap_{F\in \Psi(P)}H^+_{\cb_F}\right),
\end{align}
where $\cb_F:=(\ab_F,b_F)\in \QQ^{d+1}$. Given $v \in P \cap \ZZ^d$, we define ${\bf w}_\vb$ belonging to a free abelian group $\calF=\bigoplus_{F \in \Psi(P)} \ZZ \epsilon_F$ 
with its basis $\{\epsilon_F\}_{F \in \Psi(P)}$ as follows: 
$${\bf w}_\vb=\sum_{F \in \Psi(P)}\langle (\vb,1),\cb_F\rangle \epsilon_F=\sum_{F \in \Psi(P)}d_F(\vb) \epsilon_F.$$
Let
\begin{align}\label{groupS}
\calS= \sum_{\vb \in P \cap \ZZ^d} \ZZ{\bf w}_\vb=\left\{\sum_{F\in \Psi(P)}\ang{\vb',\cb_F}\epsilon_F : \vb' \in \ZZ\calA(P)\right\}
\end{align}
and let $\calM_P$ be the matrix whose column vectors consist of ${\bf w}_\vb$ for $\vb \in P \cap \ZZ^d$, that is, $\calM_P=\left(d_F(\vb)\right)_{F \in \Psi(P), \vb \in P\cap \ZZ^d}$. 
Then we can compute the divisor class group of $\kk[P]$ as follows:
\begin{thm}[cf. {\cite[Theorem 9.8.19]{Villa}}]\label{thm:class} Work with the same notation as above and suppose that $P$ is normal. 
Then, we have 
$$\Cl(\kk[P]) \cong \calF / \calS.$$ 
In particular, we have 
$$\Cl(\kk[P]) \cong \ZZ^r \oplus \ZZ/s_1\ZZ \oplus \cdots \oplus \ZZ/s_m\ZZ,$$ 
where $m=\dim P +1$, $r=|\Psi(P)|-m$ and $s_1,\ldots,s_m$ are positive integers appearing in the diagonal of the Smith normal form of $\calM_P$.
\end{thm}
\noindent The integers $s_1,\ldots,s_m$ are called the \textit{invariant factors} of $\calM_P$.


Let $P \subset \RR^d$ be an integral polytope and let $k_P$ be a maximal nonnegative integer satisfying the following statement:
\begin{itemize}
\item[($*$)] There exist distinct integral points $\vb_1,\ldots,\vb_{k_P} \in P\cap \ZZ^d$ and distinct facets $F_1,\ldots,F_{k_P}$ of $P$ such that 
$v_i \in \bigcap_{l=1}^{i-1} F_l$ for each $1< i \le k_P$ and $d_{F_i}(\vb_i)=1$ for each $1 \le i \le k_P$. \\
\end{itemize}

\begin{thm}[{\cite[Theorem 3.2]{M}}]\label{main1}
Let $P \subset \RR^d$ be a normal integral polytope and let $s_1,\ldots, s_m$ be the invariant factors of $\calM_P$.
Then, $s_1=\cdots=s_{k_P}=1$. In particular, $\Cl(\kk[P])$ is torsionfree if $k_P=\dim P+1$.
\end{thm}

Finally, we explain that normal toric rings whose divisor class groups are torsionfree can be described as the ring of invariants under an action of torus on polynomial rings.
Suppose that $P\subset \RR^d$ is a $d$-dimension normal integral polytope and satisfies $\ZZ\calA(P)=\ZZ^{d+1}$. 
Moreover, we assume that $\Cl(\kk[P])\cong \ZZ^r$.
We fix an isomorphism $\iota : \calF/\calS \to \ZZ^r$ and let
$\beta_F:=\iota(\epsilon_F)$ for each $F\in \Psi(P)$.
In this paper, we call $\beta_F$'s \textit{weights} of $\kk[P]$.
We consider the action of the algebraic torus $G=(\kk^{\times})^r$ on the polynomial ring $S= \kk[x_F : F\in \Psi(P)]$, which is the action induced by $g\cdot x_F=\beta_F(g)x_F$ for $g\in G$. 
Then, we have $R=S^G$ (see e.g., \cite[Theorem 2.1]{BG1}).

\medskip


\subsection{Toric rings of certain $(0,1)$-polytopes}
\label{subsection:01}
In this subsection, we introduce some properties associated with toric rings of certain $(0,1)$-polytopes.

First, we introduce certain polytopes and propositions associated with their toric rings.
\begin{itemize}
\setlength{\parskip}{0pt} 
\setlength{\itemsep}{3pt}
\item A polytopes $P$ is called a \textit{$(0,1)$-polytope} if its all vertices are $(0,1)$-vectors. 
\item A \textit{pyramid} $P \subset \RR^d$ is the convex hull of the union of a polytope $Q \subset \RR^d$ (\textit{basis} of $P$) and a point $v_0 \in \RR^d$ (\textit{apex} of $P$), where $v_0$ does not belong to $\aff(Q)$. Note that the basis of a pyramid $P$ is a facet of $P$. 
\item A polytope $P\subset \RR^d$ is \textit{simple} if each vertex of $P$ is contained in precisely $\dim P$ facets.
\end{itemize}


\begin{prop}[{\cite[Theorem 1]{KW}}]\label{Simple}
A $(0,1)$-polytope $P \subset \RR^d$ is simple if and only if it is equal to a product of $(0,1)$-simplices.
\end{prop}

\smallskip

\begin{prop}\label{pyramid}
Let $P \subset \RR^d$ be a $(0,1)$-pyramid with basis $Q$.
Then, $\kk[P]$ is the polynomial extension of $\kk[Q]$.
In particular, if $P$ is normal, then so is $Q$ and we have $$\Cl(\kk[P])\cong \Cl(\kk[Q]).$$
\end{prop}

\bigskip

Next, we introduce the certain class of $(0,1)$-polytopes, which are arising from posets.
These are called order polytopes (\cite{S86}) and play an important role in this paper. 

Let $\Pi$ be a finite poset equipped with a partial order $\preceq$. 
For a subset $I \subset \Pi$, we say that $I$ is a \textit{poset ideal} of $\Pi$ if $p \in I$ and $q \preceq p$ then $q \in I$. 
Note that $\emptyset$ is regarded as a poset ideal.
Let $\widehat{\Pi}=\Pi \sqcup \{\hat{0},\hat{1}\}$, where $\hat{0}$ (resp. $\hat{1}$) is a new minimal (resp. maximal) element not belonging to $\Pi$. 

For a poset $\Pi=\{p_1,\ldots,p_d\}$, let \begin{align*}
\calO_\Pi=\{(x_1,\ldots,x_d) \in \RR^d : \; x_i \geq x_j \text{ if } p_i \preceq p_j \text{ in }\Pi, \;\; 
0 \leq x_i \leq 1 \text{ for }i=1,\ldots,d \}.
\end{align*}
A convex polytope $\calO_\Pi$ is called the \textit{order polytope} of $\Pi$. 
It is known that $\calO_\Pi$ is a $(0,1)$-polytope and has IDP.
According to \cite{S86}, the vertices of $\calO_\Pi$ one-to-one correspond to the poset ideals of $\Pi$.
In fact, a $(0,1)$-vector $(a_1,\ldots,a_d)$ is a vertex of $\calO_\Pi$ if and only if $\{p_i\in \Pi : a_i=1\}$ is a poset ideal.
Moreover, the facets of $\calO_{\Pi}$ one-to-one correspond to the edges of the Hasse diagram of $\widehat{\Pi}$.
In fact, for $p_i, p_j \in \Pi$ with $p_i \prec p_j$, $\{p_i,p_j\}$ is an edge of the Hasse diagram of $\widehat{\Pi}$ if and only if the inequality $x_i\ge x_j$ defines a facet of $\calO_{\Pi}$,
where we let $x_i=1$ (resp. $x_j=0$) if $p_i=\hat{0}$ (resp. $p_j=\hat{1}$).

The toric ring $\kk[\calO_\Pi]$ is called the \textit{Hibi ring} of $\Pi$, which was originally introduced in \cite{H87}.
We denote the Hibi ring of $\Pi$ by $\kk[\Pi]$ (instead of $\kk[\calO_\Pi]$ for short).

Let $P$ and $Q$ be two posets with $P\cap Q=\emptyset$. 
The \textit{disjoint union} of $P$ and $Q$ is the poset $P+Q$ on $P\cup Q$ such that $x\preceq y$ in $P+Q$ if (a) $x,y\in P$ and $x\preceq y$ in $P$, or (b) $x,y\in Q$ and $x\preceq y$ in $Q$.
The \textit{ordinal sum} of $P$ and $Q$ is the poset $P\oplus Q$ on $P\cup Q \cup \{z\}$ such that $x \preceq y$ in $P\oplus Q$ if
(a) $x,y\in P$ and $x\preceq y$ in $P$, (b) $x,y\in Q$ and $x\preceq y$ in $Q$, (c) $x\in P$ and $y=z$, or (d) $x=z$ and $y\in Q$,
where $z$ is a new element which is not contained in $P\cup Q$.
By observing poset ideals of $P+Q$ and $P\oplus Q$, the following proposition holds:

\begin{prop}\label{p_sum}
Let $P$ and $Q$ be two posets with $P\cap Q=\emptyset$.
\begin{itemize}
\item[(i)] One has $\kk[P+Q] \cong \kk[P]\#\kk[Q]$.
\item[(ii)] One has $\kk[P\oplus Q]\cong \kk[Q\oplus P] \cong \kk[P]\otimes_{\kk}\kk[Q]$.
\end{itemize}
\end{prop}

We provide the description of the divisor class groups of Hibi rings, which is given in \cite{HHN}.
Let $n$ be the number of edges of the Hasse diagram of $\widehat{\Pi}$. 
Then, it is proved in \cite{HHN} that \begin{align}\label{eq:Hibi}\Cl(\kk[\Pi]) \cong \ZZ^{n-d-1}.\end{align} 
In particular, the rank of $\calO_{\Pi}$ is equal to $n-(d+1)$.

The characterization of posets $\Pi$ with $\rank \calO_\Pi \le 2$ has been investigated in \cite{N,HM3}:

\begin{thm}[{\cite[Lemma~3.2]{N} and \cite[Proposition~4.1]{HM3}}]
Let $\Pi$ be a poset. Then,
\begin{itemize}
\item[(i)] $\rank \calO_\Pi=0$ if and only if (the Hasse diagram of) $\Pi$ is a chain. 
\end{itemize}
Moreover, assume that $\kk[\Pi]$ is not a polynomial extension of a Hibi ring. Then,
\begin{itemize}
\item[(ii)] $\rank \calO_\Pi=1$ if and only $\Pi$ is the disjoint union of two chains.
\item[(iii)] $\rank \calO_\Pi=2$ if and only if $\Pi$ is one of the posets $\Pi_1,\Pi_2,\Pi_3$ or $\Pi_4$ (see Figures~\ref{poset1}, \ref{poset2}, \ref{poset3} and \ref{poset4}).
\end{itemize}
\end{thm}

\begin{figure}[h]
{\scalebox{0.8}{
\begin{minipage}{0.45\columnwidth}
\centering
{\scalebox{0.9}{
\begin{tikzpicture}[line width=0.05cm]

\coordinate (N11) at (0,0); \coordinate (N12) at (0,1);
\coordinate (N13) at (0,2); \coordinate (N15) at (0,4);
\coordinate (N16) at (0,5); \coordinate (N21) at (2,0);
\coordinate (N22) at (2,1); \coordinate (N23) at (2,2);
\coordinate (N25) at (2,4); \coordinate (N26) at (2,5);
\coordinate (N31) at (4,0); \coordinate (N32) at (4,1);
\coordinate (N33) at (4,2); \coordinate (N35) at (4,4);
\coordinate (N36) at (4,5);
\coordinate (N0) at (2,-1); \coordinate (N1) at (2,6);

\draw  (N11)--(N12); \draw  (N12)--(N13);
\draw  (N13)--(0,2.5); \draw  (0,3.5)--(N15);
\draw  (N15)--(N16); \draw[dotted]  (0,2.8)--(0,3.2);
\draw  (N21)--(N22); \draw  (N22)--(N23);
\draw  (N23)--(2,2.5); \draw  (2,3.5)--(N25);
\draw  (N25)--(N26); \draw[dotted]  (2,2.8)--(2,3.2);
\draw  (N31)--(N32); \draw  (N32)--(N33);
\draw  (N33)--(4,2.5); \draw  (4,3.5)--(N35);
\draw  (N35)--(N36); \draw[dotted]  (4,2.8)--(4,3.2);
\draw (N0)--(N11); \draw (N0)--(N21); \draw (N0)--(N31);
\draw (N1)--(N16); \draw (N1)--(N26); \draw (N1)--(N36);

\draw [line width=0.05cm, fill=white] (N11) circle [radius=0.15]; \draw [line width=0.05cm, fill=white] (N12) circle [radius=0.15]; 
\draw [line width=0.05cm, fill=white] (N13) circle [radius=0.15]; \draw [line width=0.05cm, fill=white] (N15) circle [radius=0.15]; 
\draw [line width=0.05cm, fill=white] (N16) circle [radius=0.15]; \draw [line width=0.05cm, fill=white] (N21) circle [radius=0.15];
\draw [line width=0.05cm, fill=white] (N22) circle [radius=0.15]; \draw [line width=0.05cm, fill=white] (N23) circle [radius=0.15]; 
\draw [line width=0.05cm, fill=white] (N25) circle [radius=0.15]; \draw [line width=0.05cm, fill=white] (N26) circle [radius=0.15];
\draw [line width=0.05cm, fill=white] (N31) circle [radius=0.15]; \draw [line width=0.05cm, fill=white] (N32) circle [radius=0.15]; 
\draw [line width=0.05cm, fill=white] (N33) circle [radius=0.15]; \draw [line width=0.05cm, fill=white] (N35) circle [radius=0.15]; 
\draw [line width=0.05cm, fill=white] (N36) circle [radius=0.15];
\draw [line width=0.05cm, fill=white] (N0) circle [radius=0.15]; 
\draw [line width=0.05cm, fill=white] (N1) circle [radius=0.15]; 
\node at (5.,0) {$\;$}; 
\draw [line width=0.015cm, decorate,decoration={brace,amplitude=10pt}](-0.25,-0.7) -- (-0.25,5.7) node[midway,xshift=-0.8cm] {\Large $n_1$}; 
\draw [line width=0.015cm, decorate,decoration={brace,amplitude=10pt}](1.75,-0.7) -- (1.75,5.7) node[black,midway,xshift=-0.8cm] {\Large $n_2$};
\draw [line width=0.015cm, decorate,decoration={brace,amplitude=10pt,mirror}](4.25,-0.7) -- (4.25,5.7) node[black,midway,xshift=0.8cm] {\Large $n_3$};

\end{tikzpicture}
} }
\caption{The poset $\widehat{\Pi}_1$}
\label{poset1}
\end{minipage} \; \; 
\begin{minipage}{0.45\columnwidth}
\centering
{\scalebox{0.9}{
\begin{tikzpicture}[line width=0.05cm]

\coordinate (N15) at (0,3);\coordinate (N17) at (0,5);
\coordinate (N25) at (2,3);\coordinate (N27) at (2,5);
\coordinate (Nt1) at (0,0);\coordinate (Nt3) at (0,2);
\coordinate (Nt5) at (2,0);\coordinate (Nt7) at (2,2);
\coordinate (Ntt1) at (1,2.5);
\coordinate (N0) at (1,-0.5); \coordinate (N1) at (1,5.5);

\draw (Nt1)--(0,0.5);\draw[dotted] (0,0.8)--(0,1.2); 
\draw (0,1.5)--(Nt3);\draw (Nt5)--(2,0.5); 
\draw[dotted] (2,0.8)--(2,1.2);\draw (2,1.5)--(Nt7); 
\draw (Nt3)--(Ntt1);\draw (Nt7)--(Ntt1); 
\draw (Ntt1)--(N15);\draw (Ntt1)--(N25); 
\draw (N15)--(0,3.5);\draw[dotted] (0,3.8)--(0,4.2); 
\draw (0,4.5)--(N17);\draw (N25)--(2,3.5); 
\draw[dotted] (2,3.8)--(2,4.2);\draw (2,4.5)--(N27);
\draw (N0)--(Nt1); \draw (N0)--(Nt5);
\draw (N1)--(N17); \draw (N1)--(N27);

\draw [line width=0.05cm, fill=white] (Ntt1) circle [radius=0.15]; \draw [line width=0.05cm, fill=white] (N15) circle [radius=0.15];
\draw [line width=0.05cm, fill=white] (N17) circle [radius=0.15]; \draw [line width=0.05cm, fill=white] (N25) circle [radius=0.15];
\draw [line width=0.05cm, fill=white] (N27) circle [radius=0.15]; \draw [line width=0.05cm, fill=white] (Nt1) circle [radius=0.15];
\draw [line width=0.05cm, fill=white] (Nt3) circle [radius=0.15]; \draw [line width=0.05cm, fill=white] (Nt5) circle [radius=0.15]; 
\draw [line width=0.05cm, fill=white] (Ntt1) circle [radius=0.15]; \draw [line width=0.05cm, fill=white] (Nt7) circle [radius=0.15];
\draw [line width=0.05cm, fill=white] (N0) circle [radius=0.15]; 
\draw [line width=0.05cm, fill=white] (N1) circle [radius=0.15];
\node at (3.5,0) {$\;$}; 
\draw [line width=0.015cm, decorate,decoration={brace,amplitude=10pt}](-0.25,2.6) -- (-0.25,5.4) node[midway,xshift=-0.8cm] {\Large $n_1$}; 
\draw [line width=0.015cm, decorate,decoration={brace,amplitude=10pt,mirror}](2.25,2.6) -- (2.25,5.4) node[midway,xshift=0.8cm] {\Large $n_2$}; 
\draw [line width=0.015cm, decorate,decoration={brace,amplitude=10pt}](-0.25,-0.4) -- (-0.25,2.4) node[black,midway,xshift=-0.8cm] {\Large $m_1$}; 
\draw [line width=0.015cm, decorate,decoration={brace,amplitude=10pt,mirror}](2.25,-0.4) -- (2.25,2.4) node[black,midway,xshift=0.8cm] {\Large $m_2$};

\end{tikzpicture} 
}}
\caption{The poset $\widehat{\Pi}_2$}
\label{poset2}
\end{minipage}
}}
{\scalebox{0.8}{
\begin{minipage}{0.46\columnwidth}
\centering
{\scalebox{0.9}{
\begin{tikzpicture}[line width=0.05cm]

\coordinate (N14) at (0,2.5);
\coordinate (N15) at (0,3.5);
\coordinate (N16) at (0,4.5);
\coordinate (N17) at (0,5.5);
\coordinate (N24) at (2,2.5);
\coordinate (N25) at (2,3.5);
\coordinate (N27) at (2,5.5);
\coordinate (N31) at (4,0);
\coordinate (N32) at (4,1);
\coordinate (N33) at (4,2);
\coordinate (N35) at (4,3);
\coordinate (N36) at (4,4.5); 
\coordinate (N37) at (4,5.5);
\coordinate (Nt1) at (1,0);
\coordinate (Nt3) at (1,2); 
\coordinate (N0) at (2,-1); \coordinate (N1) at (2,6.5);

\draw  (Nt3)--(N14);
\draw  (N14)--(N15);
\draw  (N15)--(0,4);
\draw[dotted]  (0,4.3)--(0,4.7);
\draw  (0,5)--(N17);
\draw  (Nt3)--(N24);
\draw  (N24)--(N25);
\draw  (N25)--(2,4);
\draw[dotted]  (2,4.3)--(2,4.7);
\draw  (2,5)--(N27); 
\draw  (Nt1)--(1,0.5);
\draw[dotted]  (1,0.8)--(1,1.3);
\draw  (1,1.5)--(Nt3); 
\draw  (N31)--(N32); 
\draw  (N32)--(4,1.9); 
\draw  (4,3.6)--(N36); 
\draw  (N36)--(N37);
\draw[dotted]  (4,2.2)--(4,3.3);
\draw (N0)--(Nt1); \draw (N0)--(N31);
\draw (N1)--(N17); \draw (N1)--(N27); \draw (N1)--(N37);

\draw [line width=0.05cm, fill=white] (Nt1) circle [radius=0.15]; \draw [line width=0.05cm, fill=white] (Nt3) circle [radius=0.15];
\draw [line width=0.05cm, fill=white] (N14) circle [radius=0.15]; \draw [line width=0.05cm, fill=white] (N15) circle [radius=0.15];
\draw [line width=0.05cm, fill=white] (N17) circle [radius=0.15]; \draw [line width=0.05cm, fill=white] (N24) circle [radius=0.15]; 
\draw [line width=0.05cm, fill=white] (N25) circle [radius=0.15]; \draw [line width=0.05cm, fill=white] (N27) circle [radius=0.15];
\draw [line width=0.05cm, fill=white] (N31) circle [radius=0.15]; \draw [line width=0.05cm, fill=white] (N32) circle [radius=0.15]; 
\draw [line width=0.05cm, fill=white] (N36) circle [radius=0.15]; \draw [line width=0.05cm, fill=white] (N37) circle [radius=0.15];
\draw [line width=0.05cm, fill=white] (N0) circle [radius=0.15]; 
\draw [line width=0.05cm, fill=white] (N1) circle [radius=0.15];
\node at (6,0) {$\;$}; 
\draw [line width=0.015cm, decorate,decoration={brace,amplitude=10pt}](-0.25,2.1) -- (-0.25,6.2) node[midway,xshift=-0.8cm] {\Large $n_1$}; 
\draw [line width=0.015cm, decorate,decoration={brace,amplitude=10pt,mirror}](2.25,2.1) -- (2.25,6.2) node[midway,xshift=0.8cm] {\Large $n_2$};
\draw [line width=0.015cm, decorate,decoration={brace,amplitude=10pt,mirror}](4.25,-0.7) -- (4.25,6.2) node[black,midway,xshift=0.8cm] {\Large $n_3$}; 
\draw [line width=0.015cm, decorate,decoration={brace,amplitude=10pt}](0.75,-0.7) -- (0.75,1.9) node[black,midway,xshift=-0.8cm] {\Large $m_1$}; 

\end{tikzpicture} }}
\caption{The poset $\widehat{\Pi}_3$}
\label{poset3}
\end{minipage}
\begin{minipage}{0.49\columnwidth}
\centering
{\scalebox{0.9}{
\begin{tikzpicture}[line width=0.05cm]
\coordinate (N11) at (0,0); \coordinate (N13) at (0,2);
\coordinate (N14) at (0,3); \coordinate (N16) at (0,5);
\coordinate (N17) at (0,6); \coordinate (N21) at (4,0);
\coordinate (N22) at (4,1); \coordinate (N24) at (4,3);
\coordinate (N25) at (4,4); \coordinate (N27) at (4,6); 
\coordinate (Ntt1) at (1,2.5); \coordinate (Ntt2) at (3,3.5);
\coordinate (N0) at (2,-1); \coordinate (N1) at (2,7);

\draw  (N11)--(0,0.5); \draw  (0,1.5)--(N13);
\draw  (N13)--(N14); \draw  (N14)--(0,3.5);
\draw  (0,4.5)--(N16); \draw  (N16)--(N17);
\draw[dotted]  (0,0.8)--(0,1.2); \draw[dotted]  (0,3.8)--(0,4.2);
\draw  (N21)--(N22); \draw  (N22)--(4,1.5);
\draw  (4,2.5)--(N24); \draw  (N24)--(N25);
\draw  (N25)--(4,4.5); \draw  (4,5.5)--(N27);
\draw[dotted]  (4,1.8)--(4,2.2); \draw[dotted]  (4,4.8)--(4,5.2);
\draw  (N13)--(Ntt1); \draw  (Ntt2)--(N25);
\draw  (Ntt1)--(1.5,2.75); \draw  (2.5,3.25)--(Ntt2);
\draw[dotted]  (1.8,2.9)--(2.2,3.1);
\draw (N0)--(N11); \draw (N0)--(N21);
\draw (N1)--(N17); \draw (N1)--(N27);

\draw [line width=0.05cm, fill=white] (N11) circle [radius=0.15]; \draw [line width=0.05cm, fill=white] (N13) circle [radius=0.15];
\draw [line width=0.05cm, fill=white] (N14) circle [radius=0.15]; \draw [line width=0.05cm, fill=white] (N16) circle [radius=0.15];
\draw [line width=0.05cm, fill=white] (N17) circle [radius=0.15]; \draw [line width=0.05cm, fill=white] (N21) circle [radius=0.15];
\draw [line width=0.05cm, fill=white] (N22) circle [radius=0.15]; \draw [line width=0.05cm, fill=white] (N24) circle [radius=0.15];
\draw [line width=0.05cm, fill=white] (N25) circle [radius=0.15]; \draw [line width=0.05cm, fill=white] (N27) circle [radius=0.15];
\draw [line width=0.05cm, fill=white] (Ntt1) circle [radius=0.15]; \draw [line width=0.05cm, fill=white] (Ntt2) circle [radius=0.15];
\draw [line width=0.05cm, fill=white] (N0) circle [radius=0.15]; 
\draw [line width=0.05cm, fill=white] (N1) circle [radius=0.15];
\node at (4.5,0) {$\;$}; 
\draw [line width=0.015cm, decorate,decoration={brace,amplitude=10pt,mirror}](0.27,1.9) -- (3.8,3.7) 
node[midway,xshift=0.4cm,yshift=-0.6cm] {\Large $m_1$}; 
\draw [line width=0.015cm, decorate,decoration={brace,amplitude=10pt}](-0.25,-0.7) -- (-0.25,1.9) node[black,midway,xshift=-0.8cm] {\Large $l_1$};
\draw [line width=0.015cm, decorate,decoration={brace,amplitude=10pt}](-0.25,2.1) -- (-0.25,6.7) node[black,midway,xshift=-0.8cm] {\Large $n_1$};
\draw [line width=0.015cm, decorate,decoration={brace,amplitude=10pt,mirror}](4.25,-0.7) -- (4.25,3.9) node[midway,xshift=0.8cm] {\Large $l_2$}; 
\draw [line width=0.015cm, decorate,decoration={brace,amplitude=10pt,mirror}](4.25,4.1) -- (4.25,6.7) node[midway,xshift=0.8cm] {\Large $n_2$};

\end{tikzpicture} }}
\caption{The poset $\widehat{\Pi}_4$}
\label{poset4}
\end{minipage}
}}
\end{figure}
Here, $n_1,n_2,n_3,m_1$ and $m_2$ are the number of edges.
Note that the Hibi ring of $\calO_{\Pi}$ with $\rank \calO_{\Pi}=0$ is isomorphic to a polynomial ring.
Moreover, it follows from Proposition~\ref{p_sum} (i) that the Hibi ring of $\calO_{\Pi}$ with $\rank \calO_{\Pi}=1$ is isomorphic to the Segre product of two polynomial rings.

Weights of the Hibi rings whose divisor class groups have rank 2 are computed In \cite[Sections~3.2 and 3.3]{N}.
The following table summarizes the weights of the Hibi rings associated with the posets in Figures~\ref{poset1}, \ref{poset2}, \ref{poset3} and \ref{poset4}.

\bigskip

\begin{center}
\begin{tabular}{ccccc} \hline \addlinespace[4pt]
Poset & $\Pi_1$ & $\Pi_2$ & $\Pi_3$ & $\Pi_4$ \\ \addlinespace[4pt] \hline 
\addlinespace[4pt]
Weights & 
\begin{tabular}{c}
$(1,0)\times n_1$ \\ \addlinespace[4pt]
$(0,1)\times n_2$ \\ \addlinespace[4pt]
$(-1,-1)\times n_3$ 
\end{tabular}&
\begin{tabular}{c}
$(1,0)\times n_1$ \\ \addlinespace[4pt]
$(-1,0)\times n_2$ \\ \addlinespace[4pt]
$(0,1)\times m_1$ \\ \addlinespace[4pt]
$(0,-1)\times m_2$
\end{tabular}&
\begin{tabular}{c}
$(1,0)\times n_1$ \\ \addlinespace[4pt]
$(-1,-1)\times n_2$ \\ \addlinespace[4pt]
$(0,1)\times n_3$ \\ \addlinespace[4pt]
$(0,-1)\times m_1$ 
\end{tabular}&
\begin{tabular}{c}
$(1,0)\times n_1$ \\ \addlinespace[4pt]
$(-1,0)\times n_2$ \\ \addlinespace[4pt]
$(-1,1)\times m_1$ \\ \addlinespace[4pt]
$(0,1)\times l_1$ \\ \addlinespace[4pt]
$(0,-1)\times l_2$
\end{tabular} \\
\addlinespace[4pt]
\hline
\end{tabular}
\end{center}

\bigskip

\noindent Here $\times n_i$, $\times m_i$ and $\times l_i$ stand for the multiplicities.


\subsection{Gale-diagrams}\label{subsec:gale}

In this subsection, we recall the notation of Gale-diagrams of a polytope, which helps us to classify the isomorphic classes of toric rings.

Throughout this subsection,
let $P\subset \RR^d$ be a $d$-dimensional polytope with the vertex set $\{\vb_1,\ldots,\vb_n\}$ and suppose that $P$ has the irreducible representation (\ref{repre}).

We consider
$$D(P)=\left\{(\alpha_1,\ldots,\alpha_n)\in \RR^n : \sum_{i\in [n]} \alpha_i\vb_i={\bf 0} \text{ and } \sum_{i\in [n]} \alpha_i=0 \right\}.$$
We can see that $D(P)$ is a $(n-d-1)$-dimensional subspace of $\RR^d$.
Let $\bb_1,\ldots,\bb_{n-d-1}$ be a basis of $D(P)$.
Moreover, for $i\in [n]$, let $\overline{\bb_i}$ denote the $i$-th column vector of the $(n-d-1)\times n$ matrix 
$\begin{pmatrix}
\bb_1  \\
\vdots  \\
\bb_{n-d-1}    
\end{pmatrix}$, that is, 
$\overline{\bb}_i=(\bb_1^{(i)},\ldots,\bb_{n-d-1}^{(i)})$, where for a vector $\vb \in \RR^n$, $\vb^{(i)}$ denotes the $i$-th coordinate of $\vb$.

Then, the $n$-tuple $(\overline{\bb_1},\ldots,\overline{\bb_n})$ is called the \textit{Gale-transform} of $\{\vb_1,\ldots,\vb_n\}$ (or of $P$). 
Furthermore, let 
$$\widehat{\bb_i} = \begin{cases} \overline{\bb_i}/||\overline{\bb_i}|| &\text{ if } \overline{\bb_i}\neq {\bf 0}, \\
{\bf 0} &\text{ if }\overline{\bb_i}= {\bf 0} \end{cases}$$
for each $i\in [n]$, where $\| \vb \|=\sqrt{\ang{\vb,\vb}}$ for a vector $\vb \in \RR^n$.
Then, the $n$-tuple $(\widehat{\bb_1},\ldots,\widehat{\bb_n})$ is called the \textit{Gale-diagram} of $\{\vb_1,\ldots,\vb_n\}$ (or of $P$).
Gale-transforms and Gale-diagrams depend on the choice of the basis of $D(P)$.

Especially, we are interested in the case $n-d-1=2$, that is, $P$ just has $d+3$ vertices.
In this case, we can draw a \textit{standard} (or \textit{reduced}) Gale-diagram from the Gale-diagram $\widehat{\bb_1},\ldots,\widehat{\bb_n}$, which consists of the unit circle centered at the origin of $\RR^2$ and diameters having at least one endpoint with multiplicity.
See \cite[Section 6.3]{Gr} for the precise way to draw it.

We say that two polytopes $Q$ and $Q'$ are \textit{combinatorially equivalent} or of the same \textit{combinatorial type} (resp. \textit{dual} to each other) if there exists a one-to-one mapping $\Phi$ between the set of all faces of $Q$ and the set of all faces of $Q'$ such that $\Phi$ is inclusion-preserving (resp. inclusion-reversing).
Note that the classes of simplicial polytopes, which are polytopes whose facets are simplices, and simple polytopes are dual to each other.

According to \cite[Sections 5.4 and 6.3]{Gr}, the following facts are known: 
\begin{itemize}
\item Two $d$-dimensional polytopes with $d + 3$ vertices are combinatorially equivalent if and only if their standard Gale-diagrams are orthogonally equivalent (i.e., isomorphic under an orthogonal linear transformation of $\RR^2$ onto itself).
\item $P$ is simplicial if and only if no diameter of the standard Gale-diagram has both endpoints. 
\end{itemize}

Clearly, all polytopes which are dual to $P$ have the same combinatorial type.
Whenever we consider dual polytopes, we focus on their combinatorial type, so we call them \textit{the} dual polytope of $P$.

For the remainder of this subsection, we will explain how to obtain a Gale-diagram of the dual polytope $P$.
The \textit{polar} $P^{\circ}$ of $P$ is defined as
$$P^{\circ}=\{ \xb\in \RR^d : \ang{\xb,\yb}\ge -1 \text{ for any $\yb\in P$}\}.$$
It is known that if the origin is an interior point of $P$, then $P^{\circ}$ is a polytope and is dual to $P$.
Moreover, $\{\ab_F/b_F\ : F\in \Psi(P)\}$ is the vertex set of $P^{\circ}$.

The following theorem means that we can get a Gale-diagram of the dual polytope of $P$ from weights of its toric ring.

\begin{thm}\label{thm:weights_Gale}
Suppose that $P$ is normal, $\ZZ\calA(P)=\ZZ^{d+1}$ and $\Cl(\kk[P])\cong \ZZ^r$.
Let $\{\beta_F\}_{F\in \Psi(P)}$ be weights of $\kk[P]$.
Then, $(\beta_F/\|\beta_F\|)_{F\in \Psi(P)}$ is a Gale-diagram of the dual polytope of $P$.
\end{thm}

Before we give the proof, we prepare the following lemma:

\begin{lem}\label{weight_relation}
Work with the same assumption and notation as Theorem~\ref{thm:weights_Gale}.
Then, we have $\sum_{F\in \Psi(P)}\beta_F^{(j)}\ab_F={\bf 0}$ and $\sum_{F\in \Psi(P)}\beta_F^{(j)}b_F=0$ for each $j\in [r]$.
\end{lem}

\begin{proof}
Let $\eb_i$ be the $i$-th unit vector of $\ZZ^{d+1}$.
Since $\eb_i\in \ZZ\calA(P)$ for each $i\in [d+1]$, $\sum_{F\in \Psi(P)}\ang{\eb_i,\cb_F}\epsilon_F$ belongs to $\calS$ from (\ref{groupS}).
Thus, by considering its image in $\ZZ^r$, we can obtain that
$\sum_{F\in \Psi(P)}\cb_F^{(i)}\beta_F={\bf 0}$ for each $i\in [d+1]$.
This fact is equivalent to $\sum_{F\in \Psi(P)}\beta_F^{(j)}\cb_F={\bf 0}$ for each $j\in [r]$.
Therefore, we obtain the desired equations.
\end{proof}

\medskip

\begin{proof}[Proof of Theorem~\ref{thm:weights_Gale}]
Let $\wb\in \RR^d$ be a vector such that the interior of $P+\ub$ has ${\bf 0}$.
Then, $Q=(P+\ub)^{\circ}$ is the dual polytope of $P$.
Moreover, $P+\ub$ has the irreducible representation
$$P+\ub=\bigcap_{F\in \Psi(P)}H^+(\ab_F;b'_F),$$
where $b'_F=b_F-\ang{\ab_F,\ub}$.
Note that $b'_F \neq 0$ for any $F\in \Psi(P)$ since ${\bf 0}$ belongs to the interior of $P+\ub$.

We show that $(b'_F\beta_F)_{F\in \Psi(P)}$ is a Gale-transform of $Q$.
For each $j\in [r]$, we define the vector $\bb_j$ of $\RR^{|\Psi(P)|}$ as $\bb_j:=(b'_F\beta_F^{(j)})_{F\in \Psi(P)}$.
Then, we can see that
$$\overline{\bb_j}=(b'_F\beta_F^{(1)},\ldots,b'_F\beta_F^{(r)})=b'_F\beta_F
\quad \text{ and }\quad \widehat{\bb_j}=\beta_F/\|\beta_F\|.$$
Hence, it is enough to show that $\bb_j$'s satisfy the following:
$$
\text{(i) for each $j\in [r]$, $\bb_j$ is in $D(Q)$ } \quad \text{ and } \quad \text{(ii) $\bb_1,\ldots,\bb_r$ form a basis of $D(Q)$.}
$$

(i) Note that the vertex set of $Q$ is $\{\ab_F / b'_F\ : F\in \Psi(P)\}$. 
From Lemma~\ref{weight_relation}, for each $j\in [r]$, 
$$\sum_{F\in \Psi(P)} b'_F\beta_F^{(j)}\cdot \ab_F/b'_F=\sum_{F\in \Psi(P)} \beta_F^{(j)}\ab_F={\bf 0}.$$
Moreover,
\begin{align*}
\sum_{F\in \Psi(P)} b'_F\beta_F^{(j)}&=\sum_{F\in \Psi(P)} b_F\beta_F^{(j)}-\sum_{F\in \Psi(P)} \ang{\ab_F,\ub}\beta_F^{(j)} \\
&=0 - \ang{\sum_{F\in \Psi(P)}\beta_F^{(j)}\ab_F,\ub}
=- \ang{{\bf 0},\ub}=0.
\end{align*}
Therefore, we have $\bb_j\in D(Q)$ for all $j\in [r]$.

(ii) Note that $D(Q)$ is a $r$-dimensional subspace of $\RR^{|\Psi(P)|}$.
Thus, it is enough to show that $\bb_1,\ldots,\bb_r$ are linearly independent.
Suppose that there exists $\ab=(a_1,\ldots,a_r) \in \RR^r$ such that 
$$a_1\bb_1 +\cdots +a_r\bb_r=0.$$
Then, for each $F\in \Psi(P)$, one has 
$$a_1(b'_F\beta_F^{(1)})+\cdots+a_r(b'_F\beta_F^{(r)})=b'_F\ang{\ab,\beta_F}=0.$$
Since $b'_F\neq 0$, we have $\ang{\ab,\beta_F}=0$  for any $F\in \Psi(P)$, equivalently, $\beta_F$'s lie on the hyperplane $H_\ab$.
This is a contradiction to the fact that $\beta_F$'s span $\calF/\calS \cong \ZZ^r$ as a semigroup (\cite[Theorem 2]{Cho}).
\end{proof}

\section{A new family of $(0,1)$-polytopes}\label{sec:new} 

In this section, we construct a new class of $(0,1)$-polytopes, denoted by $P_{n_1,\ldots,n_k}$, and study its properties.
We show that $P_{n_1,\ldots,n_k}$ is normal and torsionfree.
Moreover, we see that if $k\ge 3$, then $\kk[P_{n_1,\ldots,n_k}]$ is not isomorphic to any Hibi ring.

Let $n_1,\ldots,n_k$ be positive integers and let $d=n_1+\cdots+n_k$.
Moreover, let $\calB_n$ be the standard basis of $\RR^n$ and let $\calB_{n_1,\ldots,n_k}=\calB_{n_1}\times \cdots \times \calB_{n_k}$.
Then, we define the subset $V_{n_1,\ldots,n_k}$ of $\ZZ^d$ as
$$V_{n_1,\ldots,n_k}=\{{\bf 0}\}\cup \calB_d \cup \calB_{n_1,\ldots,n_k}$$
and define the $(0,1)$-polytope $P_{n_1,\ldots,n_k}=\conv(V_{n_1,\ldots,n_k})$.

\begin{ex}
We can see that $$P_{1,1,1}=\conv(\{(0,0,0),(1,0,0),(0,1,0),(0,0,1),(1,1,1)\})$$
and $$\kk[P_{1,1,1}]\cong \kk[x_1,x_2,x_3,x_4,x_5]/(x_2x_3x_4-x_0^2x_5).$$
In addition, we can see that
\begin{align*}
P_{1,1,1,2}=\conv(\{&(0,0,0,0,0), (1,0,0,0,0), (0,1,0,0,0), (0,0,1,0,0)\\
& (0,0,0,1,0), (0,0,0,0,1), (1,1,1,1,0),(1,1,1,0,1)\})
\end{align*}
and 
$$\kk[P_{1,1,1,2}]\cong \kk[x_1,\ldots,x_8]/(x_6x_7-x_5x_8,x_2x_3x_4x_6-x_1^3x_8,x_2x_3x_4x_5-x_1^3x_7).$$
\end{ex}

\bigskip

First, we describe the facet defining inequalities of $P_{n_1,\ldots,n_k}$.
By considering the embedding 
$$\calB_{n_p} \hookrightarrow \RR^d \quad \quad \ib_p \mapsto (\underbrace{0,\ldots \ldots,0}_{n_1+\cdots+n_{p-1}}, \; \ib_p\; , \underbrace{0,\ldots \ldots,0}_{n_{p+1}+\cdots+n_k}),$$
we regard $\ib_p\in \calB_{n_p}$ as the element of $\RR^d$.

For $p\in [k]$, let $\fb_p=(k-2){\bf 1}_p-\sum_{q\in [k]\setminus \{p\}}{\bf 1}_q$, where ${\bf 1}_p=\sum_{\ib_p\in \calB_{n_p}}\ib_p$.

\begin{prop}\label{prop:facet}
The $(0,1)$-polytope $P_{n_1,\ldots,n_k}$ has the irreducible representation:
$$P_{n_1,\ldots,n_k}=\Bigl(\bigcap_{\eb\in \calB_d}H_{\eb}^{+}\Bigr)\cap \Bigl(\bigcap_{p\in [k]}H^+(\fb_p;1)\Bigr).$$
\end{prop}

\begin{proof}
If $k=1$, then $P_{n_1}=\conv(\{{\bf 0}\}\cup \calB_{n_1})$ and we can see easily that $P_{n_1}$ has the above irreducible representation.
In what follows, suppose that $k\ge 2$.

We can find $d$ affinely independent vectors in $V_{n_1,\ldots,n_k}$ on each hyperplane. Thus, these define facets of $P_{n_1,\ldots,n_k}$.

Let $P'=\left(\bigcap_{\eb\in \calB_d}H_{\eb}^{+}\right)\cap \left(\bigcap_{p\in [k]}H^+(\fb_p;1)\right)$ and we show $P=P'$. It is clear that $P \subset P'$ since $V_{n_1,\ldots,n_k} \subset P'$. Also, each vertex of $P_{n_1,\ldots,n_l}$ is on $d$ affinely independent hyperplanes $H_\eb$ or $H(\fb_p;1)$, that is, the set of vertices of $\calP'$, denoted by $V$, contains $V_{n_1,\ldots,n_k}$. Therefore, it is enough to show that $V \subset V_{n_1,\ldots,n_k}$.

Note that if $\xb\in P'$, then $\xb\in H^+(\fb_p;1)$, i.e., $\langle \xb, \fb_p \rangle +1\ge 0$ for all $p\in[k]$.
Thus, for any $p\in [k]$, we have
\begin{align}\label{ineq}
\sum_{q\in [k]\setminus \{p\}} (\langle \xb, \fb_q \rangle +1)=-(k-1)\langle \xb, {\bf 1}_p \rangle +(k-1)\ge 0 \; \Rightarrow \; \langle \xb, {\bf 1}_p \rangle \le 1.
\end{align}
Moreover, since $\xb \in H^+_\eb$ for all $\eb \in \calB_d$, we see that $V\subset P' \subset [0,1]^d$. 

Assume that there exists an element $\vb$ in $V\setminus V_{n_1,\ldots,n_k}$.
First, suppose that $\vb \in \{0,1\}^d$.
In this case, we see that $\langle \vb, {\bf 1}_p \rangle =0 \text{ or }1$ holds for all $p\in [k]$ from (\ref{ineq}).
One of the following three cases happens;
\begin{itemize}
\item[(i)] $\langle \vb, {\bf 1}_p \rangle =0$ for all $p\in [k]$ or $\langle \vb, {\bf 1}_p \rangle =1$ for all $p\in [k]$.
\item[(ii)] $\langle \vb, {\bf 1}_p \rangle =1$ for some $p\in [k]$ and $\langle \vb, {\bf 1}_q \rangle =0$ for all $q\in [k]\setminus \{p\}$.
\item[(iii)] There exist $p_1,p_2,p_3 \in [k]$ such that  $\langle \vb, {\bf 1}_{p_1} \rangle = \langle \vb, {\bf 1}_{p_2} \rangle=1$ and  $\langle \vb, {\bf 1}_{p_3} \rangle =0$.
\end{itemize}
In the cases (i) and (ii), $\vb$ must lie in $V_{n_1,\ldots,n_k}$.
In case (iii), we have $\langle \vb, \fb_{p_3} \rangle +1 <0$, that is, $\vb \notin H^+(\fb_{p_3};1)$, a cotradiction.

Now, suppose that $\vb \notin \{0,1\}^d$. If there are two elements $\ib_p, \jb_p \in \calB_{n_p}$ with $\ang{\vb ,\ib_p}, \ang{\vb,\jb_p}>0$ for some $p\in [k]$,  then $\vb$ is not a vertex of $\calP'$.
Indeed, let \vspace{-0.15cm}
$$\vb'=\vb+\epsilon \ib_p-\epsilon \jb_p, \quad \vb''=\vb-\epsilon \ib_p+\epsilon \jb_p,  \vspace{-0.15cm}$$
where $\epsilon >0$ is sufficiently small. Then, we can check that $\vb', \vb''\in P'$ and $\vb=\frac{1}{2}(\vb'+\vb'')$.
Therefore, we assume that for each $p\in [k]$, $\ang{\vb, \ib_p}\ge 0$ for some $\ib_p\in \calB_{n_p}$ and $\ang{\vb,\jb_p}=0$ for all $\jb_p\in \calB_{n_p}\setminus \{\ib_p\}$.
Let $t=\sum_{p\in[k]}\ang{\vb,\ib_p}$.
In the case $t\le 1$, let $t_p=\ang{\vb,\ib_p}$ for $p\in [k]$ and $t_{k+1}=1-t$.
Then, we have $\vb=\sum_{p\in [k]}t_p\ib_p+t_{k+1}{\bf 0}$ and $\sum_{p\in [k+1]}t_p=1$, and hence $\vb\notin V$.
Finally, we consider the case of $t>1$.
If there exists $p\in [k]$ such that $\ang{\vb,\ib_p}=0$ then we have $\langle \vb, \fb_p \rangle+1=-\sum_{q\in[k]\setminus \{p\}}\ang{\vb,\ib_q}+1=-t+1<0$, a contradiction. 
Thus, we can see that $\ang{\vb,\ib_p}> 0$ for all $p\in [k]$.
Moreover, since $\vb$ must lie on $d$ affinely independent hyperplanes $H_\eb$ or $H(\fb_p;1)$, we have $\vb \in H(\fb_p;1)$ for all $p\in [k]$, that is, $\langle \vb, \fb_p \rangle+1=0$ holds for any $p\in [k]$.
We may assume that $\ang{\vb,\ib_1}=\min\{ \ang{\vb,\ib_p} : p\in [k]\}$, and let $t_p=
\begin{cases} \ang{\vb,\ib_p}-\ang{\vb,\ib_1} & \text{ if $p\neq 1$}, \\
\ang{\vb,\ib_1} & \text{ if $p=1$}\end{cases}$ for $p\in [k]$.
Then, we obtain 
$$\vb=\sum_{p\in [k]\setminus \{1\}}t_p\ib_p+t_1\sum_{p\in[k]}\ib_p$$
and
$$\sum_{p\in [k]}t_p=-(k-2)\ang{\vb,\ib_1}+\sum_{p\in[k]\setminus \{1\}}\ang{\vb,\ib_p}=-(k-2)\ang{\vb,{\bf 1}_1}+\sum_{p\in[k]\setminus \{1\}}\ang{\vb,{\bf 1}_p}=-\langle \vb, \fb_1 \rangle=1.$$
Therefore, $v$ is not a vertex of $P'$, and hence $V$ coincides with $V_{n_1,\ldots,n_k}$.

\end{proof}

\begin{ex}
We consider $P_{1,1,1}$. 
Then, 
\begin{align*}
x_i&\ge 0 \quad \text{($i=1,2,3$)}, \\
\ang{\fb_p, \xb}+1= x_p-(x_{q_1}+x_{q_2})+1 &\ge 0 \quad \text{($p=1,2,3$ and $\{p,q_1,q_2\}=\{1,2,3\}$)}
\end{align*}
are the facet defining inequalities of $P_{1,1,1}$.

Moreover, we consider $P_{1,1,1,2}$.
Then, the following are the facet defining inequalities of $P_{1,1,1,2}$:
\begin{align*}
x_i &\ge 0 \quad \text{($i=1,\ldots,5$)}, \\
2x_p-(x_{q_1}+\cdots +x_{q_4})+1 &\ge 0 \quad \text{($p=1,2,3$ and $\{p,q_1,\ldots,q_4\}=\{1,\ldots,5\}$) and} \\
2(x_4+x_5)-(x_1+x_2+x_3)+1 &\ge 0.
\end{align*}
\end{ex}

\bigskip

Next, we investigate the initial ideal of the toric ideal $I_{P_{n_1,\ldots,n_k}}$ with respect to a monomial order and provide its \Gro basis, which allows us to study the normality.
For the fundamental materials on initial ideals and \Gro bases, consult, e.g., \cite{HHO}.

For $x_{\vb} \in T=\kk[x_\vb : \vb\in P_{n_1,\ldots,n_k} \cap \ZZ^d]$, we denote by $x_{\ib_1,\ldots,\ib_k}$ instead of $x_{(\ib_1,\ldots,\ib_k)}$ for $(\ib_1,\ldots,\ib_k) \in \calB_{n_1,\ldots, n_k}$.
Let $<$ denote the reverse lexicographic order on $T$ induced by the ordering of the variables as follows:

\begin{itemize}
\item $x_\eb < x_{\ib_1,\ldots,\ib_k}$  for any $\eb \in \{{\bf 0}\}\cup\calB_d$ and any $(\ib_1,\ldots,\ib_k) \in \calB_{n_1,\ldots,n_k}$; \vspace{0.2cm}
\item For $\eb, \eb'\in \{{\bf 0}\}\cup\calB_d$ (resp. for $(\ib_1,\ldots,\ib_k), (\jb_1,\ldots,\jb_k) \in \calB_{n_1,\ldots,n_k}$), $x_\eb < x_{\eb'}$ (resp. $x_{\ib_1,\ldots,\ib_k} < x_{\jb_1,\ldots,\jb_k}$) if and only if $\eb < \eb'$ (resp. $(\ib_1,\ldots,\ib_k)<(\jb_1,\ldots,\jb_k)$), which means that the leftmost nonzero component of the vector $\eb'-\eb$ (resp. $(\jb_1-\ib_1,\ldots,\jb_k-\ib_k)$) is positive.
\end{itemize}

\medskip

Moreover, let $\calG_{n_1,\ldots,n_k}$ be the sets of the following binomials in $T$: 
\begin{itemize}
\item[(b1)] $x_{\ib_1}x_{\ib_2} \cdots x_{\ib_k}-x_{{\bf 0}}^{k-1}x_{\ib_1,\ldots,\ib_k}$ for $(\ib_1,\ldots,\ib_k)\in \calB_{n_1,\ldots,n_k}$;
\item[(b2)] $x_{\jb_p}x_{\ib_1,\ldots,\ib_p,\ldots,\ib_k}-x_{\ib_p}x_{\ib_1,\ldots,\jb_p,\ldots,\ib_k}$ for $(\ib_1,\ldots,\ib_k)\in \calB_{n_1,\ldots,n_k}$, $\jb_p \in \calB_{n_p}$ with $\ib_p<\jb_p$;
\item[(b3)] $x_{\ib_1,\ldots,\ib_k}x_{\jb_1,\ldots,\jb_k}-x_{\ib'_1,\ldots,\ib'_k}x_{\jb'_1,\ldots,\jb'_k}$ for $(\ib_1,\ldots,\ib_k), (\jb_1,\ldots,\jb_k) \in \calB_{n_1,\ldots,n_k},$
\end{itemize}
where we define 
$\ib'_p=\begin{cases} \ib_p \; \text{ if } x_{\ib_p}<x_{\jb_p}, \\ \jb_p \; \text{ else } \end{cases}$
and define $\jb'_p$ as satisfying  $\{\ib'_p,\jb'_p\}=\{\ib_p,\jb_p\}$.
Note that each leading term of these binomials is the initial monomial with respect to $<$.
Furthermore, these binomials belong to the toric ideal of $P_{n_1,\ldots,n_k}$, that is, $\calG_{n_1,\ldots,n_k} \subset I_{P_{n_1,\ldots,n_k}}$.

\bigskip

\begin{prop}\label{prop:gro} 

Let the notation be the same as above. Then, $\calG_{n_1,\ldots,n_k}$ is a \Gro basis of $I_{\calP_{n_1,\ldots,n_k}}$ with respect to $<$.

\end{prop}

\noindent To show this proposition, we use the following lemma:

\begin{lem}[{\cite[Theorem 3.11]{HHO}}]\label{lem_gro}
Let $I\subset S=\kk[x_1,\ldots,x_n]$ be the toric ideal of an integral polytope and $\calG=\{g_1,\ldots,g_s\}$ the set of binomials in $I$. Fix a monomial order $<$ on $S$ and let $\ini_{<}(\calG)$ be the ideal of $S$ generated by the initial monomials $\ini_{<}(g_1),\ldots,\ini_<(g_s)$, that is, $\ini_{<}(\calG)=(\{ \ini_{<}(g) : g \in \calG\}) $. Then, the following conditions are equivalent: 
\begin{itemize}
\item[(i)] $\calG$ is a \Gro basis of $I$ with respect to $<$. 
\item[(ii)] For monomials $u, v \in S$, if $u\notin \ini_{<}(\calG)$, $v\notin \ini_{<}(\calG)$ and $u\neq v$ then $u-v\notin I$.
\end{itemize}
\end{lem}

\bigskip

\begin{proof}[Proof of Proposition~\ref{prop:gro}]
We show that $\calG_{n_1,\ldots,n_k}$ satisfies the condition (ii) in Lemma~\ref{lem_gro}.
Let $u$ be a monomial in $T$.
Then, $u$ can be written by 
$$x_{{\bf 0}}^ax_{\eb_1}x_{\eb_2}\cdots x_{\eb_s}x_{\ib_{11},\ldots,\ib_{1k}}\cdots x_{\ib_{t1},\ldots,\ib_{tk}},$$
where $a\in \ZZ_{\ge 0}$, $\eb_1,\ldots,\eb_s \in \calB_d$ and $\ib_{1p},\ldots, \ib_{tp} \in \calB_{n_p}$ for each $p\in [k]$.
Let $M_p=\{\eb_1,\ldots,\eb_s\} \cap \calB_{n_p}$ and $N_p=|M_p|$ for $p\in [k]$.
Then, if $u \notin \ini_{<}(\calG_{n_1,\ldots,n_k})$, we can see that:

\begin{itemize}
\item[(a)] there exists $p \in[k]$ such that $M_p=\emptyset$ from (b1);
\item[(b)] for each $p\in [k]$, we have $\eb_i \le \ib_{lp}$ for any $\eb\in M_p$ and for any $l\in [t]$ from (b2);
\item[(c)] $\ib_{1p} \le \ib_{2p} \le \cdots \le \ib_{tp}$ for all $p\in [k]$ by permuting $x_{\ib_{11},\ldots,\ib_{1k}},\ldots,x_{\ib_{t1},\ldots,\ib_{tk}}$ from (b3).
\end{itemize} 

Let $\phi_{P_{n_1\ldots,n_k}}(u)=t_1^{r_1}t_2^{r_2}\cdots t_d^{r_d} t_0^{r_0}$. 
Now, we may assume that $M_1=\emptyset$.
Then, we can see that $t=r_1+r_2+\cdots +r_{n_1}$.
Similarly, $N_p$ and $a$ can also be represented by $r_i$'s. 
Moreover, it follows from (b), (c) that $\eb_i$ and $\ib_{1p},\ldots,\ib_{tp}$ can be determined uniquely from $r_0,\ldots,r_d$.
Therefore, we can recover $u$ from $t_1^{r_1}t_2^{r_2}\cdots t_d^{r_d} t_0^{r_0}$. This is equivalent to (ii) in Lemma~\ref{lem_gro}.
\end{proof}

Now, we give the main theorem in this section.

\begin{thm}\label{thm:new_pro}
The $(0,1)$-polytope $P_{n_1\ldots,n_k}$ has the following properties:
\begin{itemize}
\item[(i)] 
$P_{n_1,\ldots,n_k}$ has IDP.
\item[(ii)] $\kk[P_{n_1,\ldots,n_k}]$ is Gorenstein if and only if $n_1=n_2=\cdots =n_k$.
\item[(iii)] $\Cl(\kk[P_{n_1,\ldots,n_k}]) \cong \ZZ^{k-1}$.
\item[(iv)] If $k\ge 3$, then $\kk[P_{n_1,\ldots,n_k}]\notin {\bf Order}_{k-1}$ for any $n_1,\ldots,n_k\in \ZZ_{>0}$.
\end{itemize}
\end{thm}

\begin{proof}
(i) It follows from Proposition~\ref{prop:gro} that the initial ideal of $I_{n_1,\ldots,n_k}$ is squarefree, and hence $P_{n_1,\ldots,n_k}$ possesses a regular unimodular triangulation (cf. \cite[Theorem~4.17]{HHO}).
This implies that $P_{n_1,\ldots,n_k}$ is normal. 
Moreover, $\ZZ \calA(P_{n_1,\ldots,n_k})$ coincides with $\ZZ^{d+1}$ since $\{{\bf 0}\}\cup\calB_d\subset P_{n_1,\ldots,n_k}$, and hence $P_{n_1,\ldots,n_k}$ has IDP.

(ii) Since $\kk[P_{n_1,\ldots,n_k}]$ is a normal affine semigroup ring, the canonical module $\omega_{\kk[P_{n_1,\ldots,n_k}]}$ is isomorphic to the module generated by all monomials whose exponent vector is a lattice point in 
$\Bigl(\bigcap_{\eb\in \calB_d}H^+((\eb,0);1)\Bigr)\cap \Bigl(\bigcap_{p\in [k]}H^+((\fb_p,1);1)\Bigr)$.
By the parallel translation by $(1,\ldots,1,\alpha)$ for some integer $\alpha$, we can see that it is also isomorphic to the module generated by all monomials whose exponent vector is a lattice point in 
$\Bigl(\bigcap_{\eb\in \calB_d}H^+((\eb,0);0)\Bigr)\cap \Bigl(\bigcap_{p\in [k]}H^+((\fb_p,1);m_p)\Bigr)$,
where $m_p=1-\alpha-(k-2)n_p+\sum_{q\in[k]\setminus \{p\}} n_q$.

Thus, if $\kk[P_{n_1,\ldots,n_k}]$ is Gorenstein, then the equality $m_1=\cdots=m_k=0$ must hold.
This implies $n_1=\cdots=n_k$.
Conversely, if $n_1,\ldots,n_k$ are equal, then we can see that $\omega_{\kk[P_{n_1,\ldots,n_k}]}$ is isomorphic to $\kk[P_{n_1,\ldots,n_k}]$ by setting $\alpha=n_1+1$, that is, $\kk[P_{n_1,\ldots,n_k}]$ is Gorenstein.

(iii) Let $\vb_1,\ldots,\vb_d$ be vectors in $\ZZ^d$ satisfying $\{\vb_1,\ldots,\vb_d\}=\calB_d$ and let $\vb_{d+1}={\bf 0}$.
Moreover, for $i\in [d]$, let $F_i$ be the facet of $P_{n_1,\ldots,n_k}$ defined by  $H(\vb_i;0)$ and let $F_{d+1}$ be the facet of $P_{n_1,\ldots,n_k}$ defined by $H(\fb_1;1)$.
Then, we can see that these sequences satisfy the statement ($*$) and we have $k_{P_{n_1,\ldots,n_k}}=d+1$.
Therefore, we have $\Cl(\kk[P_{n_1,\ldots,n_k}])\cong \ZZ^{k-1}$ from Theorem~\ref{main1}.

(iv) Since $\calG_{n_1,\ldots,n_k}$ is a \Gro basis, the binomials (b1), (b2) and (b3) generate $I_{P_{n_1,\ldots,n_k}}$.
Consider a minimal generators $\calG'$ of $I_{P_{n_1,\ldots,n_k}}$ contained these binomials.
If $\calG'$ does not contain any binomial $x_{\ib_1}x_{\ib_2} \cdots x_{\ib_k}-x_{{\bf 0}}^{k-1}x_{\ib_1,\ldots,\ib_k}$ in (b1), then there exists an expression:
$$x_{\ib_1}x_{\ib_2} \cdots x_{\ib_k}-x_{{\bf 0}}^{k-1}x_{\ib_1,\ldots,\ib_k}=\sum_{i=1}^s \alpha_i \xb^{\wb_i} f_i,$$
where $\alpha_i\in \ZZ$, $\xb^{\wb_i}$ is a monomial of the polynomial ring $\kk[x_{\vb} : \vb \in P\cap \ZZ^d]$ and $f_i$'s are binomials in (b2) or (b3) (cf. \cite[Lemma~3.7]{HHO}).
However, it is impossible because the variable $x_{\ib_1,\ldots,\ib_k}$ for some $(\ib_1,\ldots,\ib_k)\in V_{n_1,\ldots,n_k}$ must appear in the terms of $f_i$.
Therefore, $\calG'$ contains $x_{\ib_1}x_{\ib_2} \cdots x_{\ib_k}-x_{{\bf 0}}^{k-1}x_{\ib_1,\ldots,\ib_k}$ for some $\ib_1,\ldots,\ib_k$. 
This implies that $I_{n_1,\ldots}$ cannot be generated by quadratic binomials.

On the other hand, the toric ideals of Hibi rings are generated by quadratic binomials (\cite{H87}).
Thus, $\kk[P_{n_1,\ldots,n_k}]$ is not isomorphic to any Hibi ring.
\end{proof}

Finally, we compute weights of $\kk[P_{n_1,\ldots,n_k}]$.
Let $F_{\eb}$ denote the facet $P_{n_1,\ldots,n_k}\cap H(\eb;0)$ for $\eb\in \calB_d$ and let $F_i$ denote the facet $P_{n_1,\ldots,n_k}\cap H(\fb_p;1)$ for $p\in [k]$.
From (\ref{groupS}) and Theorem~\ref{thm:new_pro} (i), the following elements in $\calF$ belong to $\calS$:
\begin{align}\label{weight_ibp}
\sum_{F\in \Psi(P_{n_1,\ldots,n_k})}\ang{(\ib_p,0),\cb_F}\epsilon_F=
\epsilon_{F_{\ib_p}}+(k-2)\epsilon_{F_p} - \sum_{q\in [k]\setminus \{p\}}\epsilon_{F_q}
\end{align}
for each $p\in [k]$ and $\ib_p\in \calB_{n_p}$ and
\begin{align}\label{weight_01}
\sum_{F\in \Psi(P_{n_1,\ldots,n_k})}\ang{({\bf 0},1),\cb_F}\epsilon_F=\sum_{p\in [k]}\epsilon_{F_p}.
\end{align}
We consider the map $\iota : \calF/\calS \to \ZZ^{k-1}$; let 
$\iota(\epsilon_{F_i})=\eb_i$ for $i\in [k-1]$, where $\eb_i$ denotes the $i$-th unit vector of $\ZZ^{k-1}$.
This induces an isomorophism $\iota : \calF/\calS \to \ZZ^{k-1}$ and we can calculate the remaining weight from (\ref{weight_ibp}) and (\ref{weight_01}):
\begin{align*}
\iota(\epsilon_{F_k})&=-(\eb_1+\cdots+\eb_{k-1}); \\
\iota(\epsilon_{F_{\ib_p}})&=-(k-1)\eb_p \quad \text{ for each $p\in [k-1]$ and any $\ib_p\in \calB_{n_p}$}; \\
\iota(\epsilon_{F_{\ib_k}})&=(k-1)(\eb_1+\cdots+\eb_{k-1}) \quad \text{ for any $\ib_k\in \calB_{n_k}$}.
\end{align*}
In particular, we can get the weights of the case $k=3$ as follows:
\begin{align}\label{weight_new}
(1,0), \quad (0,1),\quad  (-1,-1),\quad (-2,0)\times n_1, \quad (0,-2)\times n_2,\quad (2,2)\times n_3.
\end{align}


\section{Answers to the problems}\label{sec:ans}
Throughout this section, let $P\subset \RR^d$ be a $(0,1)$-polytope.
We give complete or partial answers to Problems~\ref{p:normal}, \ref{p:torsion}, \ref{p:class} and \ref{p:comb} in each case;
\begin{center}
(r1) \; $\rank P=0$ or $1$, \quad \quad  (r2) \; $\rank P=2$, \quad \quad  (r3) \; $\rank P\ge 3$.
\end{center}

\subsection{Case (r1)}

First, we discuss the case $\rank P=0$, this case is trivial.
Notice that $P$ has rank $0$ if and only if $P$ is a simplex.
Moreover, it is known that the toric ring of $d$-dimensional $(0,1)$-simplex is isomorphic to the polynomial ring with $d+1$ variables over $\kk$ (cf. \cite[Lemma 3.1.5]{Villa}).

Clearly, polynomial rings are normal and their divisor class groups are torsionfree.
Moreover, for two $(0,1)$-simplices $P$ and $P'$, one has $\kk[P]\cong \kk[P']$ if and only if $P$ and $P'$ are combinatorially equivalent (equivalently, they have the same number of vertices).
Therefore, we get the following theorem:

\begin{prop}\label{main:rank0}
All $(0,1)$-polytopes with rank $0$ are normal and torsionfree.
Moreover, the following relationship holds:
$${\bf (0,1)}_0={\bf Order}_0=\{\kk[x_1,\ldots,x_k ] : k\in \ZZ_{>0}\}.$$
Furthermore, for two $(0,1)$-polytopes $P_1$ and $P_2$ with $\rank P_i=0$, they have the same combinatorial type if and only if $\kk[P_1]\cong \kk[P_2]$.
\end{prop}

\bigskip

Next, we discuss the case $\rank P=1$.
This case has been investigated in \cite{M} as follows:


\begin{thm}[{\cite[Theorem~3.7]{M}}]\label{thm:M}
The following conditions are equivalent:
\begin{itemize}
\item[(i)] $P$ has rank $1$, that is, $P$ has just $\dim P+2$ facets;
\item[(ii)] $\kk[P]$ is isomorphic to the Segre product of two polynomial rings $\kk[x_1,\ldots,x_{n+1}]$ and $\kk[y_1,\ldots,y_{m+1}]$ for some $n,m \in \ZZ_{>0}$ or its polynomial extension;
\item[(iii)] $P$ is normal and $\Cl(\kk[P]) \cong \ZZ$.
\end{itemize}
\end{thm}

Therefore, $(0,1)$-polytopes which have rank $1$ are normal and torsionfree.
Moreover, the equivalence (i) and (ii) imply that for two $(0,1)$-polytopes $P$ and $P'$ with $\rank P=\rank P'=1$, one has $\kk[P]\cong \kk[P']$ if and only if $P$ and $P'$ are combinatorially equivalent.
Hence, the following theorem holds:
\begin{thm}\label{main:rank1}
All $(0,1)$-polytopes with rank $1$ are normal and torsionfree.
Moreover, the following relationship holds:
$${\bf (0,1)}_1={\bf Order}_1=\{(\kk[x_1,\ldots,x_{n+1}]\#\kk[y_1,\ldots,y_{m+1}])\otimes_{\kk}\kk[z_1,\ldots,z_{l-1}] : n,m,l\in \ZZ_{>0}\}.$$
Furthermore, for two $(0,1)$-polytopes $P_1$ and $P_2$ with $\rank P_i=1$, they have the same combinatorial type if and only if $\kk[P_1]\cong \kk[P_2]$.
\end{thm}

\bigskip


\subsection{Case (r2)}

By Proposition~\ref{pyramid}, in what follows, we may assume that $P$ is not pyramidal.

By Theorem~\ref{thm:weights_Gale}, we can obtain the standard Gale-diagrams of the dual polytopes of the order polytopes which have rank 2 and $P_{n_1,n_2,n_3}$.
We draw the standard Gale-diagrams as follows; the dual polytopes of $\calO_{\Pi_1},\calO_{\Pi_2},\calO_{\Pi_3},\calO_{\Pi_4}$ and $P_{n_1,n_2,n_3}$ correspond to the standard Gale-diagrams $\Gale_1,\Gale_2,\Gale_3,\Gale_4$ and $\Gale_5$, respectively.


\begin{figure}[h]
{\scalebox{0.8}{
\begin{minipage}{0.50\columnwidth}
\centering
{\scalebox{0.9}{
\begin{tikzpicture}[line width=0.05cm]

\coordinate (L1) at (0,2); \coordinate (L2) at (0,4); \coordinate (L3) at (0,0); 
\coordinate (L4) at (-1.732,3.0); \coordinate (L5) at (1.732,1.0); 
\coordinate (L6) at (1.732,3.0); \coordinate (L7) at (-1.732,1.0);

\draw (L2)--(L3); 
\draw (L4)--(L5); 
\draw (L6)--(L7);

\draw [line width=0.05cm] (L1) circle [radius=2.0];
\draw [line width=0.05cm, fill=white] (L2) circle [radius=0.15] node[above] {\Large $n_1$};
\draw [line width=0.05cm, fill=white] (L5) circle [radius=0.15] node[below right] {\Large $n_3$}; 
\draw [line width=0.05cm, fill=white] (L7) circle [radius=0.15] node[below left] {\Large $n_2$}; 


\end{tikzpicture}
}}
\caption{\; \\The Gale-diagram $\Gale_1$}
\label{type1}
\end{minipage}
\begin{minipage}{0.50\columnwidth}
\centering
{\scalebox{0.9}{
\begin{tikzpicture}[line width=0.05cm]

\coordinate (L1) at (0,2); \coordinate (L2) at (0,4); \coordinate (L3) at (0,0); 
\coordinate (L4) at (2,2); \coordinate (L5) at (-2,2);

\draw (L2)--(L3); 
\draw (L4)--(L5); 

\draw [line width=0.05cm] (L1) circle [radius=2.0];
\draw [line width=0.05cm, fill=white] (L2) circle [radius=0.15] node[above] {\Large $n_1$};
\draw [line width=0.05cm, fill=white] (L3) circle [radius=0.15] node[] at (0,-0.4) {\Large $n_2$};
\draw [line width=0.05cm, fill=white] (L4) circle [radius=0.15] node[right] {\Large $m_2$};
\draw [line width=0.05cm, fill=white] (L5) circle [radius=0.15] node[left] {\Large $m_1$};

\end{tikzpicture}
}}
\caption{\; \\The Gale-diagram $\Gale_2$}
\label{type2}
\end{minipage}
}}
\end{figure}


\begin{figure}[h]
{\scalebox{0.8}{
\begin{minipage}{0.50\columnwidth}
\centering
{\scalebox{0.9}{
\begin{tikzpicture}[line width=0.05cm]

\coordinate (L1) at (0,2); \coordinate (L2) at (0,4); \coordinate (L3) at (0,0); 
\coordinate (L4) at (-1.732,3.0); \coordinate (L5) at (1.732,1.0); 
\coordinate (L6) at (1.732,3.0); \coordinate (L7) at (-1.732,1.0);

\draw (L2)--(L3); 
\draw (L4)--(L5); 
\draw (L6)--(L7);

\draw [line width=0.05cm] (L1) circle [radius=2.0];
\draw [line width=0.05cm, fill=white] (L2) circle [radius=0.15] node[above] {\Large $n_1$};
\draw [line width=0.05cm, fill=white] (L4) circle [radius=0.15] node[above left] {\Large $m_1$};
\draw [line width=0.05cm, fill=white] (L5) circle [radius=0.15] node[below right] {\Large $n_3$}; 
\draw [line width=0.05cm, fill=white] (L7) circle [radius=0.15] node[below left] {\Large $n_2$}; 


\end{tikzpicture}
}}
\caption{\; \\The Gale-diagram $\Gale_3$}
\label{type3}
\end{minipage}
\begin{minipage}{0.50\columnwidth}
\centering
{\scalebox{0.9}{
\begin{tikzpicture}[line width=0.05cm]

\coordinate (L1) at (0,2); \coordinate (L2) at (0,4); \coordinate (L3) at (0,0); 
\coordinate (L4) at (-1.732,3.0); \coordinate (L5) at (1.732,1.0); 
\coordinate (L6) at (1.732,3.0); \coordinate (L7) at (-1.732,1.0);

\draw (L2)--(L3); 
\draw (L4)--(L5); 
\draw (L6)--(L7);

\draw [line width=0.05cm] (L1) circle [radius=2.0];
\draw [line width=0.05cm, fill=white] (L2) circle [radius=0.15] node[above] {\Large $n_1$};
\draw [line width=0.05cm, fill=white] (L3) circle [radius=0.15] node[] at (0,-0.4) {\Large $n_2$};
\draw [line width=0.05cm, fill=white] (L4) circle [radius=0.15] node[above left] {\Large $l_1$};
\draw [line width=0.05cm, fill=white] (L5) circle [radius=0.15] node[below right] {\Large $l_2$}; 
\draw [line width=0.05cm, fill=white] (L7) circle [radius=0.15] node[below left] {\Large $m_1$};

\end{tikzpicture}
}}
\caption{\; \\The Gale-diagram $\Gale_4$}
\label{type4}
\end{minipage}
}}
\end{figure}


\begin{figure}[h]
{\scalebox{0.8}{
\begin{minipage}{1.0\columnwidth}
\centering
{\scalebox{0.9}{
\begin{tikzpicture}[line width=0.05cm]

\coordinate (L1) at (0,2);
\coordinate (L2) at (0,4); \coordinate (L3) at (0,0); 
\coordinate (L4) at (-1.732,3.0); \coordinate (L5) at (1.732,1.0); 
\coordinate (L6) at (1.732,3.0); \coordinate (L7) at (-1.732,1.0);

\draw (L2)--(L3); 
\draw (L4)--(L5); 
\draw (L6)--(L7);

\draw [line width=0.05cm] (L1) circle [radius=2.0];
\draw [line width=0.05cm, fill=white] (L2) circle [radius=0.15] node[above] {\Large $n_1$};
\draw [line width=0.05cm, fill=white] (L3) circle [radius=0.15] node[] at (0,-0.4) {\Large $1$};
\draw [line width=0.05cm, fill=white] (L4) circle [radius=0.15] node[above left] {\Large $1$};
\draw [line width=0.05cm, fill=white] (L5) circle [radius=0.15] node[below right] {\Large $n_3$}; 
\draw [line width=0.05cm, fill=white] (L6) circle [radius=0.15] node[above right] {\Large $1$}; 
\draw [line width=0.05cm, fill=white] (L7) circle [radius=0.15] node[below left] {\Large $n_2$};

\end{tikzpicture}
}}
\caption{\; \\The Gale-diagram $\Gale_5$}
\label{type5}
\end{minipage}
}}
\end{figure}
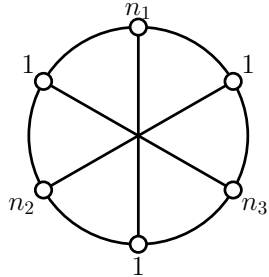

\begin{thm}\label{thm:type1}
The following are equivalent:
\begin{itemize}
\item[(i)] The standard Gale-diagram of the dual polytope of $P$ is orthogonally equivalent to $\Gale_1$.
\item[(ii)] The toric ring of $P$ is isomorphic to the Segre product of three polynomial rings over $\kk$.
\end{itemize}
In particular, let $P_1$ and $P_2$ be two $(0,1)$-polytopes which have a standard Gale-diagram $\Gale_1$, then $P_1$ and $P_2$ are combinatorially equivalent if and only if $\kk[P_1]\cong \kk[P_2]$.
\end{thm}

\begin{proof}
(i) $\Rightarrow$ (ii) : Since no diameter of the standard Gale-diagram has both endpoints, the dual polytope of $P$ is simplicial, that is, $P$ is simple.
Therefore, $P$ is the product of just three $(0,1)$-simplices from Proposition~\ref{Simple}, and hence $\kk[P]$ is isomorphic to the Segre product of three polynomial rings.

(ii) $\Rightarrow$ (i) : It follows from Proposition~\ref{p_sum} that the Segre product of three polynomial rings can be realized as the Hibi ring of a poset $\Pi_1$.
Its weights have been already given in Section~\ref{subsection:01}.
By Theorem~\ref{thm:weights_Gale}, we can obtain the standard Gale-diagram of the dual polytope of $P$ which is orthogonally equivalent to $\Gale_1$.

The last statement follows from the equivalence (i) and (ii).

\end{proof}

\begin{thm}\label{thm:type2}
The following are equivalent:
\begin{itemize}
\item[(i)] The standard Gale-diagram of the dual polytope of $P$ is orthogonally equivalent to $\Gale_2$.
\item[(ii)] The toric ring of $P$ is isomorphic to $R_1\otimes_{\kk} R_2$, where $R_i$ is the Segre product of two polynomial rings over $\kk$.
\end{itemize}
In particular, let $P_1$ and $P_2$ be two $(0,1)$-polytopes which have a standard Gale-diagram $\Gale_2$, then $P_1$ and $P_2$ are combinatorially equivalent if and only if $\kk[P_1]\cong \kk[P_2]$.
\end{thm}

\begin{proof}
(i) $\Rightarrow$ (ii) : Since the standard Gale-diagram of the dual polytope of P is orthogonally equivalent to $\Gale_2$, $P$ is combinatorially equivalent to $\calO_{\Pi_2}$.
Therefore, there exists a one-to-one mapping $\Phi$ between the set of all faces of $\calO_{\Pi_2}$ and the set of all faces of $P$ such that $\Phi$ is inclusion preserving.

Let $\ub_1,\ldots,\ub_n$ denote the vertices of $\calO_{\Pi_2}$ corresponding to the poset ideals of $\Pi_2$ including the element which is comparable with any other element of $\Pi_2$, and let $\vb_1,\ldots,\vb_m$ denote the remaining vertices of $\calO_{\Pi_2}$.
Notice that any facet $F$ of $\calO_{\Pi_2}$ contains $\{\ub_1,\ldots,\ub_n\}$ or $\{\vb_1,\ldots,\vb_m\}$, that is,
$d_F(\ub_i)=0 \text{ for all $i\in [n]$, or } d_F(\vb_j)=0 \text{ for all $j\in [m]$}.$
Moreover, we can see that $O_1=\conv(\{\ub_1,\ldots,\ub_n\})$ and $O_2=\conv(\{\vb_1,\ldots,\vb_m\})$ are $(0,1)$-polytopes with $\rank O_1=\rank O_2=1$.

Let $\wb_i=\Phi(\ub_i)$, $\zb_j=\Phi(\vb_j)$, $Q_1=\Phi(O_1)=\conv(\{\wb_1,\ldots,\wb_n\})$ and $Q_2=\Phi(O_2)=\conv(\{\zb_1,\ldots,\zb_m\})$.
We show that $\kk[P]\cong \kk[Q_1]\otimes_{\kk}\kk[Q_2]$, equivalently, $I_P=RI_{Q_1}+RI_{Q_2}$, where $R=\kk[x_\vb : \vb\in P\cap \ZZ^d]$.
Since $I_P\supset RI_{Q_1}$ and $I_P\supset RI_{Q_2}$, we have $I_P\supset RI_{Q_1}+RI_{Q_2}$. 
To prove the reverse inclusion, it is enough to show that a binomial
$b=x_{\wb_{i_1}}\cdots x_{\wb_{i_p}}x_{\zb_{j_1}}\cdots x_{\zb_{j_q}}-x_{\wb_{h_1}}\cdots x_{\wb_{h_s}}x_{\zb_{g_1}}\cdots x_{\zb_{g_t}} \in I_P$ belongs to $RI_{Q_1}+RI_{Q_2}$. 

In this situation, we have
$$\wb'_{i_1}+\cdots+\wb'_{i_p}+\zb'_{j_1}+\cdots+\zb'_{j_q}=\wb'_{h_1}+\cdots+\wb'_{h_s}+\zb'_{g_1}+\cdots+\zb'_{g_t},$$
where for $\vb\in \ZZ^d$, we define $\vb'=(\vb,1)$.

Since for any $F'\in \Psi(P)$, $d_{F'}(\wb_i)=0 \text{ for all $i\in [n]$, or } d_{F'}(\zb_j)=0 \text{ for all $j\in [m]$}$,
we can see that $\ang{\sum_{k\in [p]}\wb'_{i_k}-\sum_{l\in [s]}\wb'_{h_l},\cb_{F'}}=0$ for all $F' \in \Psi(P)$.
Indeed, if $F'$ contains $\wb_i$ for all $i\in [n]$, then
\begin{align*}
\ang{\sum_{k\in [p]}\wb'_{i_k}-\sum_{l\in [s]}\wb'_{h_l},\cb_{F'}}
&=\sum_{k\in [p]}\ang{\wb'_{i_k},\cb_{F'}}-\sum_{l\in [s]}\ang{\wb'_{h_l},\cb_{F'}} \\
&=\sum_{k\in [p]}d_{F'}(\wb_{i_k})-\sum_{l\in [s]}d_{F'}(\wb_{h_l})=0.
\end{align*}
Moreover, if $F'$ contains $\zb_j$ for all $j\in [m]$, then
\begin{align*}
\ang{\sum_{k\in [p]}\wb'_{i_k}-\sum_{l\in [s]}\wb'_{h_l},\cb_{F'}}
&=\ang{\sum_{k\in [p]}\wb'_{i_k}+\sum_{k\in [q]}\zb'_{j_k}-\sum_{l\in [s]}\wb'_{h_l}-\sum_{l\in [t]}\zb'_{g_l},\cb_{F'}} \\
&=\ang{{\bf 0},\cb_{F'}}=0.
\end{align*}
The homomorphism from $\ZZ\calA(P)$ onto $\calS$ given by $\vb' \mapsto \sum_{F\in \Psi(P)}\ang{\vb',\cb_F}\epsilon_F$ is an isomorphism (cf. \cite[Proposition~9.8.17]{Villa}).
Therefore, we have $\sum_{k\in [p]}\wb'_{i_k}-\sum_{l\in [s]}\wb'_{h_l}={\bf 0}$.
Similarly, we have $\sum_{k\in [q]}\zb'_{j_k}-\sum_{l\in [t]}\zb'_{g_l}={\bf 0}$, and hence 
\begin{align*}
b&=x_{\wb_{i_1}}\cdots x_{\wb_{i_p}}(x_{\zb_{j_1}}\cdots x_{\zb_{j_q}}-x_{\zb_{g_1}}\cdots x_{\zb_{g_t}})+
x_{\zb_{g_1}}\cdots x_{\zb_{g_t}}(x_{\wb_{i_1}}\cdots x_{\wb_{i_p}}-x_{\wb_{h_1}}\cdots x_{\wb_{h_s}}) \\
&\in RI_{Q_1}+RI_{Q_2}.
\end{align*}

Since $Q_1$ and $Q_2$ have rank $1$, it follows from Theorem~\ref{thm:M} that $\kk[Q_1]$ and $\kk[Q_2]$ are the Segre products of two polynomial rings, we get desired. 

(ii) $\Rightarrow$ (i) : 
By Proposition~\ref{p_sum} (i) and (ii), $\kk[P]$ can be realized as the Hibi ring of a poset $\Pi_2$.
From its weights given in Section~\ref{subsection:01} and Theorem~\ref{thm:weights_Gale}, we can see that the standard Gale-diagram of the dual polytope of $P$ is orthogonally equivalent to $\Gale_2$.

The last statement follows from the equivalence (i) and (ii).
\end{proof}

\bigskip

These theorems do not give a complete answer to our problems.
We are left with the following questions:

\begin{q}\label{ques}
For any $(0,1)$-polytope $P$ with $\rank P=2$, is $\kk[P]$ isomorphic to a Hibi ring or $\kk[P_{n_1,n_2,n_3}]$ for some $n_1,n_2,n_3\in \ZZ_{>0}$?
In other words, does the relationship ${\bf (0,1)}_2={\bf Order}_2\sqcup \{\kk[P_{n_1,n_2,n_3}] : n_1,n_2,n_3\in \ZZ_{>0}\}$ hold?
\end{q}

\begin{q}
For any $(0,1)$-polytope $P$ with $\rank P=2$, is the standard Gale-diagram of the dual polytope of $P$ orthogonally equivalent to one of the $\Gale_i$'s?
Also, let $P$ and $P'$ be two $(0,1)$-polytopes whose dual polytopes have the standard Gale-diagrams $\Gale_3$, $\Gale_4$ or $\Gale_5$, then does the combinatorial equivalence of $P$ and $P'$ imply the isomorphism of their toric rings?
\end{q}
\noindent If Question~\ref{ques} has a positive answer, then all $(0,1)$-polytopes with rank $2$ are normal and torsionfree.

\subsection{Case (r3)}
Finally, we discuss the normality, torsionfreeness and classification in the case $\rank P \ge  3$.
In fact, unlike the previous cases, these properties are not guaranteed.

\begin{prop}\label{prop:nonnormal}
For any positive integer $r\ge 3$, there exists a non-normal $(0,1)$-polytope $P$ with $\rank P=r$.
\end{prop}
\begin{proof}
Let
$$Q_1=\conv(\{(0,0,0,0), (1,1,0,0), (1,0,1,0), (0,1,1,0), (0,0,0,1), (1,1,1,1)\}).$$
Then, we can see that $Q_1$ has rank 3 and is not normal.
Indeed, we can check that $\dim Q_1=4$ and $Q_1$ has 8 facets by using Magma (\cite{magma}), thus $\rank Q_1=3$.
Moreover, one has $(1,1,1,0,2)\in \ZZ\calA(Q_1)\cap\RR_{\ge 0}\calA(Q_1)$ while $(1,1,1,0,2)\notin \ZZ_{\ge 0}\calA(Q_1)$.

Furthermore, $Q_1\times [0,1]^{r-3}$, where $[0,1]^d$ denotes the $d$ dimensional unit cube, is also a non-normal $(0,1)$-polytope and its rank is equal to $r$.
\end{proof}

\begin{prop}\label{prop:nontorsion}
For any positive integer $r\ge 3$, there exists a non-torsionfree normal $(0,1)$-polytope $P$ with $\rank P=r$.
\end{prop}
\begin{proof}
Let
$$Q_2=\conv(\{(0,0,0,0), (1,0,0,0),(0,1,0,1), (0,0,1,1), (1,1,1,0), (1,1,1,1)\}).$$
Then, we can see that $Q_2$ has rank 3 and is normal but not torsionfree.
Indeed, we can see that $\ZZ\calA(Q_2)=\ZZ^5$ and $Q_2$ has IDP by using Magma, so $Q_2$ is normal.
Moreover, we can compute $\calM_{Q_2}$ and its Smith normal form as follows:
$$\calM_{Q_2}=\begin{pmatrix}
1 &0 &2 &0 &0 &0 \\
1 &0 &0 &2 &0 &0 \\
1 &0 &0 &0 &2 &0 \\
1 &0 &0 &0 &0 &2 \\
0 &1 &2 &0 &0 &0 \\
0 &1 &0 &2 &0 &0 \\
0 &1 &0 &0 &2 &0 \\
0 &1 &0 &0 &0 &2 
\end{pmatrix}
\longrightarrow
\begin{pmatrix}
1 &0 &0 &0 &0 &0 \\
0 &1 &0 &0 &0 &0 \\
0 &0 &2 &0 &0 &0 \\
0 &0 &0 &2 &0 &0 \\
0 &0 &0 &0 &2 &0 \\
0 &0 &0 &0 &0 &0 \\
0 &0 &0 &0 &0 &0 \\
0 &0 &0 &0 &0 &0 
\end{pmatrix}.
$$
Therefore, we have $\Cl(\kk[Q_2])\cong \ZZ^3\oplus (\ZZ/2\ZZ)^3$.

Moreover, $P:=Q_2\times [0,1]^{r-3}$ is also a normal $(0,1)$-polytope and its rank is equal to $r$.
We can calculate $\calM_P$ and its Smith normal form as follows:

$$\calM_P=\begin{pmatrix}
\calM_{Q_2} & \cdots & \calM_{Q_2} \\
 & A_{r-3} &  
\end{pmatrix}
\longrightarrow
\begin{pmatrix}
1 &  &       &  &  &\\
   &\ddots &       &  & &\\
   &  &1 &  & &\\
   &  &       &2 & &\\
   &  &       &  &2 &\\
   & & & & &
\end{pmatrix},
$$
where we define $A_1=
\begin{pmatrix}
1 &1 &1 &1 &1 &1 &0 &0 &0 &0 &0 &0 \\
0 &0 &0 &0 &0 &0 &1 &1 &1 &1 &1 &1
\end{pmatrix}$ and 
$$A_n=\begin{pmatrix}
 & A_{n-1} & & & A_{n-1} & \\
1 &\cdots & 1 & 0  &\cdots  &0 \\
0 &\cdots  &0 & 1 &\cdots &1
\end{pmatrix}
$$
\smallskip
for $n\ge 2$.
Therefore, we obtain $\Cl(\kk[P])\cong \ZZ^r\oplus (\ZZ/2\ZZ)^2$.
\end{proof}

It seems so hopeless to classify the isomorphism classes in this case because even in the case of the Hibi ring, it is difficult to give a complete classification.
In addition, the method using Gale-diagrams is no longer useful.
In fact, there exist two $(0,1)$-polytopes which have the same combinatorial type such that their toric rings are not isomorphic to each other.

\begin{prop}\label{prop:noniso}
For any positive integer $r\ge 3$, there exist two $(0,1)$-polytopes $P$ and $P'$ with the same combinatorial type and $\rank P=\rank P'=r$ such that their toric rings are not isomorphic to each other.
\end{prop}
\begin{proof}
Actually, $Q_1$ and $Q_2$ appearing in Propositions~\ref{prop:nonnormal} and \ref{prop:nontorsion} satisfy those conditions.
Magma confirms that these are combinatorially equivalent.
On the other hand, $\kk[Q_2]$ is normal, but $\kk[Q_1]$ is not.
Therefore, these are not isomorphic.

The same holds for $Q_1\times [0,1]^{r-3}$ and $Q_2\times [0,1]^{r-3}$.
\end{proof}


\end{document}